\documentclass{amsart}[14pt]
\usepackage{amsmath}
\usepackage{amssymb}
\usepackage{pifont}
\usepackage{epsfig}
\usepackage[colorlinks=true,linkcolor=blue,citecolor=red]{hyperref}

\usepackage{latexsym,amsfonts,amssymb,amsmath,amsthm}
\usepackage{graphicx}

\usepackage[usenames,dvipsnames]{color}
\usepackage{ulem}
\usepackage[toc,page]{appendix}

\usepackage{graphicx}
\usepackage{caption}
\usepackage{subcaption}
\usepackage{amscd}
\usepackage{amsmath}
\usepackage{latexsym}
\usepackage{amsfonts}
\usepackage{amssymb}
\usepackage{color}
\usepackage{multicol}
\usepackage{soul}
\usepackage{bm}
\usepackage{comment}

\newtheorem{theorem}{Theorem}[section]

\newtheorem{claim}{Claim}

\newtheorem{definition}[theorem]{Definition}

\newtheorem{lemma}[theorem]{Lemma}
\newtheorem{corollary}[theorem]{Corollary}

\newtheorem{proposition}[theorem]{Proposition}
\newtheorem{remark}[theorem]{Remark}

\newtheorem{prop}[theorem]{Proposition}

\numberwithin{equation}{section}

\newcommand\bes{\begin{eqnarray}}
\newcommand\ees{\end{eqnarray}}
\newcommand\bess{\begin{eqnarray*}}
\newcommand\eess{\end{eqnarray*}}

\begin{document}
\setlength{\baselineskip}{15pt}
\title[Entire Solutions of Diffusive Lotka-Volterra System]
{Entire Solutions of Diffusive Lotka-Volterra System}

\author{King-Yeung Lam$^{1}$}
\author{Rachidi B. Salako$^{1}$}
\author{Qiliang Wu$^{2}$}

\address{$^1$Department of Mathematics,
Ohio State University, Columbus, OH 43210}
\address{$^2$Department of Mathematics,
Ohio University, Athens, OH 45701}

\email{lam.184@math.ohio-state.edu (K-Y. Lam), salako.7@osu.edu (R. B. Salako), wuq@ohio.edu (Q. Wu)}



\def\R{{\mathbb R}}

\maketitle

\begin{abstract}
This work is concerned with the existence of entire solutions of the diffusive Lotka-Volterra competition system
\begin{equation}\label{eq:abstract}
\begin{cases}
u_{t}= u_{xx}+  u(1 - u-av), & \qquad \  x\in\mathbb{R} \cr
v_{t}= d v_{xx}+  rv(1 -v- bu), & \qquad \  x\in\mathbb{R} 
\end{cases}
\end{equation}
where $d,r,a$, and $b$ are positive constants with $a\neq 1$ and $b\neq 1$. We prove the existence of some entire solutions $(u(t,x),v(t,x))$ of \eqref{eq:abstract} corresponding to $(\Phi_{c}(\xi),0)$ at $t=-\infty$ (where $\xi=x-ct$ and $\Phi_c$ is a traveling wave solution of the scalar Fisher-KPP defined by the first equation of \eqref{eq:abstract} when $a=0$). Moreover, we also describe the asymptotic behavior of these entire solutions as $t\to+\infty$. We prove existence of new entire solutions for both the weak and strong competition case. In the weak competition case, we prove the existence of a class of entire solutions that forms a 4-dimensional manifold.

%
\end{abstract}



\section{Introduction}\label{s:1}
The Lotka-Voltera competition systems are frequently used to describe the  population dynamics of several competing species in their spatial domain. In this work we consider the following diffusive Lotka-Volterra competition system  of two species in the unbounded domain $\mathbb{R}$:
\begin{equation}\label{main system}
\begin{cases}
u_t=u_{xx}+u(1-u-av), \\
v_{t}=dv_{xx}+rv(1-bu-v),\quad  
\end{cases}
\end{equation}
where $a,b,d$, and $r>0$ are positive {constants}. The solutions $u(t,x)$ and $v(t,x)$ of \eqref{main system} represent {respectively} the densities of the two competing species at time $t$ and location $x\in\mathbb{R}$. Since densities must be nonnegative, only nonnegative solutions of \eqref{main system} will be of interest in this paper. 
It is well known that the asymptotic dynamics of solutions to \eqref{main system} depends delicately on  the choice of the initial distribution $(u_0(x),v_0(x))$ and the range of the parameters $a$ and $b$. Consider, for instance, the kinetic ODE system of \eqref{main system}, that is,
\begin{equation}\label{ODE}
\begin{cases}
U_t=U(1-U-aV),  \quad & \text{ for } t >0, \\
V_{t}=rV(1-bU-V), \quad &  \text{ for } t >0,\\
\end{cases}
\end{equation}
with arbitrary positive initial conditions  $U_0> 0$ and $V_0>0$, the following results are well known.
\begin{enumerate}
\item[{\bf (1)}] If $0<a,b<1$, then every solution of \eqref{ODE}  converges to the positive {equilibrium} ${\bf e}_*:=(\frac{1-a}{1-ab},\frac{1-b}{1-ab})$.
\item[{\bf (2)}] If $a,b>1$, then the behavior of solution of \eqref{ODE} depends on the choice of initial data $(u_0,v_0)$.
\item[{\bf (3)}] If $0<a<1<b$, then  every solution of \eqref{ODE} converges to  ${\bf e}_1:=(1,0)$.
\item[{\bf (4)}] If $0<b<1<a$, then every solution of \eqref{ODE} converges to  ${\bf e}_2:=(0,1)$.
\end{enumerate}

Since case {\bf (4)} can be handled similarly to {\bf (3)}, we shall henceforth consider only cases {\bf (1)} to {\bf (3)}. 

Two basic questions concerning the dynamics of \eqref{main system} are the characterization of spreading speeds of solutions and the existence of nontrivial entire solutions. By an entire solution we mean a classical solution $(u(t,x),v(t,x))$ that satisfies \eqref{main system} for $(t,x) \in \mathbb{R}^2$. Traveling waves solutions, i.e. translational invariant solutions of the form $(u(t,x),v(t,x))=(\varphi(x-ct),\psi(x-ct))$
with some appropriate boundary conditions  on $(\varphi,\psi)$ at $\pm \infty$, is an important class of entire solutions. 

Recently, Liu et al \cite{LLL2019}, Carr\`ere \cite{Car2018}, and  Gerardin and Lam \cite{LeoLam2018} studied spreading speeds of solutions of the Cauchy problem \eqref{main system} in cases {\bf (1)}, {\bf (2)}, and {\bf (3)} respectively. Among others, in case {\bf (3)}, Girardin and Lam \cite{LeoLam2018}  showed that ``if the weaker competitor is also the faster one, then it is able to evade the stronger and slower competitor by invading first into unoccupied territories. The pair of speeds depends on the initial values. If these are null in a right half-line, then the first speed is the KPP speed of the fastest competitor and the second speed is given by an exact formula depending on the first speed and on the minimal speed of traveling waves connecting the two semi-extinct equilibria. " Similar results were also established by Carr\`ere \cite{Car2018} in case {\bf (2)}, Lam et. al \cite{LLL2019} in case {\bf (1)}.

 From a dynamical point of view, large time behaviors of solutions have a strong connection with the existence of entire solutions. It is the aim of this paper to establish the existence of some entire solutions of \eqref{main system} which, when $t \to \infty$, behaves similarly as those solutions to Cauchy problems studied in \cite{Car2018,LeoLam2018,LLL2019}. In a sense, the entire solutions established in this paper are attractors to which the solutions to the Cauchy problems studied in \cite{Car2018,LeoLam2018,LLL2019}.

\medskip

\subsection*{Statement of  Main Results.}

In this subsection we state our main results on the existence of entire solutions of \eqref{main system}. We first recall some known results from related literature. 

When $a=0$, the system \eqref{main system} is decoupled and its first equation reduces to 
\begin{equation}
\label{kpp-eq}
u_{t}= u_{xx}+ u(1-u),\quad  x\in\R,
\end{equation}
which is referred to as {the} Fisher-KPP equation \cite{Fis,KPP}.  Among important solutions of \eqref{kpp-eq} are traveling wave solutions connecting the constant solutions $1$ and $0$. In fact, for each $c \geq 2$ the equation \eqref{kpp-eq} admits traveling wave solutions $u(t,x)=\Phi_c(x-ct)$ connecting $1$ and $0$, where $ \Phi_c(\xi)$ denote the unique (up to translation) solution to 
  \begin{equation}\label{tw-eq}
\begin{cases}
-c\Phi'=\Phi''+\Phi(1-\Phi)\quad \xi\in\R,\cr
\Phi(-\infty)=1, \quad \Phi(\infty)=0,
\end{cases}
\end{equation}
 and has no such traveling wave
solutions of slower speed $c<2$; see \cite{Fis, KPP, Wei1} for more details. Moreover, the stability of these traveling wave solutions of \eqref{kpp-eq} connecting $1$ and $0$ has also been studied; see \cite{Bra,Faye2019,Sat,Uch} and references therein.

 Specifically, let $\tau_c:=\frac{1}{2}(c-\sqrt{c^2-4})$ and $\widetilde{\tau}_c:=\frac{1}{2}(\sqrt{c^2+4}-c)$.  For $c>2$, the profile $\Phi_c$  is decreasing and can be chosen so that  for every $\tau\in (\tau_c,\min\{2\tau_c,1\}) $, there exist $M_c\gg 1$ and $x_{c}\gg 0$ such that 
\begin{equation}\label{decay-estimate for phi_c}
0< e^{-\tau_c x}-M_ce^{-\tau x}\leq \Phi_c(x)\leq e^{-\tau_c x},\,\,\, \text{for}\ x\ge x_{c}.
\end{equation}
Note also that the wave profile $\Phi_2$ is decreasing and for every $c>2$ there is $K_c\gg 1$ and $\widetilde{x}_c\gg 0$ such that 
\begin{equation}\label{decay-estimate for phi_2}
\Phi_2(x)\leq K_{c}\Phi_c(x)\quad\ \forall\ x>\widetilde{x}_c.
\end{equation}
Furthermore, for every $c\geq 2$, there is $\widetilde{M}_c>0$ such that 
\begin{equation}\label{decay-estimate for phi_c-at-infty}
\Phi_{c}(x)=1-\widetilde{M}_ce^{\widetilde{\tau}_cx}+o(e^{\widetilde{\tau}_cx})\quad \text{as}\ x\to-\infty,
\end{equation}
where we recall $\widetilde{\tau}_c=\frac{1}{2}(\sqrt{c^2+4}-c)$.
Observe also that for every $c\geq 2\sqrt{dr} $, the profile $\Psi_{c}(x)=\Phi_{\frac{c}{\sqrt{rd}}}(\sqrt{\frac{r}{d}}x)$ is the unique (up to translation) solution to 
  \begin{equation}\label{tw-eq for v}
\begin{cases}
-c\Psi'=d\Psi''+r\Psi(1-\Psi),\quad x\in\R,\cr
\Psi(-\infty)=1, \quad \Psi(\infty)=0.
\end{cases}
\end{equation}

\medskip

There are also many works on traveling wave solutions of the system \eqref{main system}. We refer our readers to \cite{Faye2019b,Ga,GaJo,Ho,Ka1,Ka2,Ka3,LLW1,TaFi} and the references therein for details.     For appropriate choice of $c\in\R$, we abuse the notation slightly and say that $(\varphi_c,\psi_c): \mathbb{R} \to [0,1]^2$ is  a traveling wave solution to \eqref{main system} with speed $c$, provided it satisfies
\begin{equation}\label{eq:TW-eq1} 
\begin{cases}
0=\varphi_c''+c\varphi_c'+\varphi_c(1-\varphi_c-a\psi_c),\cr
0=d\psi_c''+c\psi_c'+r\psi_c(1-\psi_c-b\varphi_c).\cr
\end{cases} 
\end{equation}
Moreover,  we introduce notations of the minimal speeds of traveling waves of the system \eqref{eq:TW-eq1}, depending on the range of parameters and boundary conditions at infinity.
 
\begin{itemize}
\item[-] If $0<a<1<b$, we denote $C_1\geq 2\sqrt{1-a}$ the minimal speed of solutions of \eqref{eq:TW-eq1} with boundary conditions $$(\varphi_c,\psi_c)(-\infty)={\bf e}_1\quad \text{and}\quad (\varphi_c,\psi_c)(\infty)={\bf e}_2.$$ 
\item[-] If $a,b>1$, we denote $C_{uv} \in \mathbb{R}$ the unique speed of solutions of \eqref{eq:TW-eq1} with boundary conditions $$(\varphi_{uv},\psi_{uv})(-\infty)={\bf e}_2\quad \text{and}\quad (\varphi_{uv},\psi_{uv})(\infty)={\bf e}_1.$$ 
\item[-] If $0<a,b<1$, we denote $C^*_1 \geq 2\sqrt{dr(1-b)}$ the minimal speed of solutions of \eqref{eq:TW-eq1} with boundary conditions $$(\varphi_c,\psi_c)(-\infty)={\bf e}_*\quad \text{and}\quad (\varphi_c,\psi_c)(\infty)={\bf e}_2.$$
\item[-] If $0<a,b<1$, we denote $C^*_2\geq 2\sqrt{1-a}$ the minimal speed of solutions of \eqref{eq:TW-eq1} with boundary conditions $$(\varphi_c,\psi_c)(-\infty)={\bf e}_*\quad \text{and}\quad (\varphi_c,\psi_c)(\infty)={\bf e}_1.$$
\end{itemize}

\vspace{.25 cm}
\begin{center}
\begin{tabular}{ |c|c|c| } 
 \hline
 Minimal speed & Range of $a, b$ & Boundary conditions at infity \\  \hline
 $C_1$ & $0<a<1<b$ & $(\varphi_c,\psi_c)(-\infty)={\bf e}_1, \quad (\varphi_c,\psi_c)(\infty)={\bf e}_2$ \\ \hline
$C_{uv}$ & $a,b>1$ & $(\varphi_c,\psi_c)(-\infty)={\bf e}_2, \quad (\varphi_c,\psi_c)(\infty)={\bf e}_1$ \\  \hline
  $C^*_1$ & $0<a,b<1$ &   $ (\varphi_c,\psi_c)(-\infty)={\bf e}_*,\quad (\varphi_c,\psi_c)(\infty)={\bf e}_2$ \\ \hline
   $C^*_2$ & $0<a,b<1$ &   $(\varphi_c,\psi_c)(-\infty)={\bf e}_*,\quad (\varphi_c,\psi_c)(\infty)={\bf e}_1$ \\ \hline
\end{tabular}
\end{center}
\vspace{.25 cm}
There are very few works on entire solutions of \eqref{main system}; see \cite{Guo2019,MoTa2009}. Morita and Tachibana in \cite{MoTa2009} established the existence of some entire solutions of \eqref{main system} of merging fronts type under the cases (2) and (3), where as $t \to -\infty$ the solution looks like two traveling waves connecting ${\bf e}_1$ and ${\bf e}_2$ coming towards each other, and  as $t \to +\infty$ the solution converges to either ${\bf e}_1$ or ${\bf e}_2$ uniformly in $x \in \mathbb{R}$. In \cite{Guo2019}, the authors treated the bistable case (2), and showed the existence of traveling fronts that is a combination of three or four merging traveling fronts. 
In this paper, we will construct three new types of entire solutions, which are different from those established in \cite{Guo2019,MoTa2009}. More specifically, all of these new entire solutions originate from the traveling front $\bm{\Phi}_c(x-ct):=(\Phi_c(x-ct),0)$  as $t \to -\infty$, and, as $t \to+\infty$, evolve to distinctive diverging fronts, whose profiles rely heavily on the competency of each species; that is, case $(1)-(3)$ results in different long time dynamics of these entire solutions. In particular, for the weak competition case (1), it is shown that the set of new entire solutions form a 4-dimensional manifold, with a limiting case discussed in Theorem \ref{MT2}. 
The general structure of entire solution of \eqref{main system} remains an interesting and challenging research direction. We refer, however, to \cite{Hamel1999,Hamel2001} for progress on the Fisher-KPP equation.

To state our main results, we first define, for every $d,c,r>0$, the auxiliary function $g_{d,c,r}$ as follows
\begin{equation}\label{g-definition}
\begin{aligned}
g_{d,c,r}: & \ [0, \frac{c}{2d}] & \longrightarrow&\quad  [r-\frac{c^2}{4d},r] &\\
&\quad\lambda & \longmapsto &\quad d\lambda^2-c\lambda  + r.&
\end{aligned}
\end{equation}
For given $\lambda\in\Lambda_{d,c,r,b}:=\{\lambda\in (0, \frac{c}{2d}) \ : \ g_{d,c,r}(\lambda)>r\max\{0, 1-b\}\}$, we introduce the speed
$$ 
c_v:=d\lambda+ \frac{r}{\lambda}.
$$ 
For $0<a<1$ and $c\ge 2$, we set $ \widetilde{\lambda}_{acc}=\frac{1}{2}\left(c_v-\sqrt{(c_v-2\tau_c)^2+4a} \right)$ and 
$$ 
\widehat{c}_{acc}:=
\begin{cases}
\widetilde{\lambda}_{acc}+(1-a)\widetilde{\lambda}^{-1}_{acc},&   \widetilde{\lambda}_{acc}\leq\sqrt{1-a}, \cr
2\sqrt{1-a},& \text{otherwise},
\end{cases}
$$
and introduce various speeds
\[
c_{u, 1}:=\max\{C^*_1,\widehat{c}_{acc}\} \,\,\, \text{if} \,\, b<1  \quad \text{and}\quad  c_{u,2}:=\max\{C_1,\widehat{c}_{acc}\}\,\,\, \text{if}\,\, b>1.
\] 
In addition, if $0<b<1$, we set $\widetilde{\lambda}:=\min\left\{\sqrt{\frac{r(1-b)}{d}},\frac{1}{2d}\left[\sqrt{c^2+4d(g_{d,c,r}(\lambda)+r(b-1))}-c\right]\right\}$ and introduce the speed
\[
\widetilde{c}_v=\max\{C^*_2,d\widetilde{\lambda}+r(1-b)\widetilde{\lambda}^{-1}\}
\]
Denoting the $L^\infty$-norm of a function ${\bf u}(x)=(u(x),v(x)): \R\to\R^2$ as $\|{\bf u}\|_{\infty}:=\sup_{x\in\R}\{|u(x)|, |v(x)|\}$ and the $L^1$-norm of a vector ${\bf u}=(u,v)$ as $|{\bf u}|_1:=|u|+|v|$, we state our main results on the existence of entire solutions of \eqref{main system}. 

\medskip
\begin{theorem}[Divergent type]\label{Main Tm}
Given $a, b, d, r>0$ and $c\geq 2$, the Lotka-Volterra system \eqref{main system} admits the traveling wave solution 
\[
\bm{\Phi}_c(x-ct):=(\Phi_c(x-ct), 0),
\]
from which ``originates" 
a family of entire solutions ${\bf u}_\lambda$, parametrized by $\lambda\in( 0 ,\sqrt{r/d})$ such that $g_{d,c,r}(\lambda) > r\max\{1-b,0\}$, denoted as
\[
{\bf u}_\lambda(t,x):=(u_\lambda(t,x), v_\lambda(t,x))\in C^{1,2}(\R\times\R^2, \R^2),
\]
in the sense that 
\[
\lim_{t\to-\infty}\|{\bf u}_{\lambda}(t,\cdot)-\bm{\Phi}_c(\cdot-ct)\|_\infty=0.
\]
Moreover, the ``destiny"--long time dynamics as $t\to+\infty$--of these entire solutions depends essentially on the ``competency" of each species; that is, the range of $a$ and $b$. More specifically, we have the following cases.
\begin{enumerate}
\item If $0<a,b<1$, then we have
\begin{subequations}\label{e:MT1.1}
\begin{align}
\limsup_{t\to\infty}\sup_{x\leq -(\widetilde{c}_v+\varepsilon)t}|{\bf u}_\lambda(t,x)-{\bf e}_1|_1=&0, \quad \forall\ 0<\varepsilon\ll 1,\label{MT1.1.1}\\
\limsup_{t\to\infty}\sup_{-(\widetilde{c}_v-\varepsilon)t\leq x\leq (c_{u,1}-\varepsilon)t}|{\bf u}_\lambda(t,x)-{\bf e}_*|_1=&0, \quad \forall\ 0<\varepsilon\ll 1, \label{MT1.1.2}\\
\limsup_{t\to\infty}\sup_{(c_{u,1}+\varepsilon)t\leq x\leq (c_v-\varepsilon)t}|{\bf u}_\lambda(t,x)-{\bf e}_2|_1=&0,\quad \forall\ 0<\varepsilon\ll 1,\label{MT1.1.3}\\
\limsup_{t\to\infty}\sup_{x\geq (c_v+\varepsilon)t}|{\bf u}_\lambda(t,x)|_1=&0, \quad{\bf \forall\ \varepsilon >0}. \label{MT1.1.4}
\end{align}
\end{subequations}
 In fact, denoting $\bm{\Psi}_{c_v}(\xi):=(0,\Psi_{c_v}(\xi))$,  there exists $h_0 \in \mathbb{R}$ such that \eqref{MT1.1.3} and \eqref{MT1.1.4} can be improved to
\begin{equation}\label{MT1.1.5}
\limsup_{t\to\infty}\sup_{x >(c_{u,1}+\varepsilon)t} |{\bf u}_\lambda(t,x)- \bm{\Psi}_{c_v}(x-c_vt-h_0)|_1 =0.
\end{equation}

\item If $a,b>1$, then there exists $h_0,h_1 \in \mathbb{R}$ such that 
$$
 \limsup_{t\to\infty}\sup_{x\leq (c_v - \varepsilon)t} |{\bf u}_\lambda(t,x)- (\varphi_{uv}(x-C_{uv}t - h_1),\min\{\psi_{uv}(x-C_{uv}t - h_1), \Psi_{c_v}(x-c_v t - h_0)\} )|_1=0,
$$
where $(\varphi_{uv}, \psi_{uv})$ is the traveling wave solution connecting ${\bf e}_1$ at $-\infty$ to ${\bf e}_2$ at $+\infty$, with speed $C_{uv}$. In particular, we have convergence to homogeneous states in coordinates moving at speed that is different from $C_{uv}$ and $c_v$, i.e.
\begin{subequations}\label{e:MT1.2}
\begin{align}
 \limsup_{t\to\infty}\sup_{x\leq (C_{uv} - \varepsilon)t} |{\bf u}_\lambda(t,x)- {\bf e}_1|_1=&0,\quad \forall\ \varepsilon>0,\label{e:MT1.2.1}\\
\limsup_{t\to\infty}\sup_{(C_{uv}+\varepsilon)t\leq x\leq (c_v-\varepsilon)t}|{\bf u}_\lambda(t,x)-{\bf e}_2|_1=&0,\quad \forall\ 0<\varepsilon\ll 1, \label{e:MT1.2.2}\\
\limsup_{t\to\infty}\sup_{x\geq (c_v+\varepsilon)t}|{\bf u}_\lambda(t,x)|_1=&0, \quad \forall\ \varepsilon>0.\label{e:MT1.2.3}
\end{align}

\end{subequations}
\item If $0<a<1<b$, then 
\begin{subequations}\label{e:MT1.3}
\begin{align}
\limsup_{t\to\infty}\sup_{x\leq (c_{u,2}-\varepsilon)t}|{\bf u}_\lambda(t,x)-{\bf e}_1|_1=&0,\quad \forall\ \varepsilon>0, \label{e:MT1.3.1}\\
\limsup_{t\to\infty}\sup_{(c_{u,2}+\varepsilon)t\leq x\leq (c_v-\varepsilon)t}|{\bf u}_\lambda(t,x)-{\bf e}_2|_1=&0,\quad \forall\ 0<\varepsilon\ll 1, \label{e:MT1.3.2}\\
\limsup_{t\to\infty}\sup_{x\geq (c_v+\varepsilon)t}|{\bf u}_\lambda(t,x)|_1=&0, \quad \forall \varepsilon>0. \label{MT1.3.3}
\end{align}
\end{subequations}
In fact, \eqref{e:MT1.3.2} and \eqref{MT1.3.3} can be improved to
\begin{equation}\label{MT1.3.4}
\limsup_{t\to\infty}\sup_{x >(c_{u,1}+\varepsilon)t} |{\bf u}_\lambda(t,x)- \bm{\Psi}_{c_v}(x-c_v t - h_0))|_1 =0, \quad \text{ for some }h_0.
\end{equation}
\end{enumerate}
\end{theorem}
\begin{remark}
For $\lambda \in (0,\sqrt{r/d})$ and $c> 2\sqrt{dr\min\{b,1\}}$,
$$
g_{d,c,r}(\lambda) > r\max\{1-b,0\} \quad \text{ iff }\quad 0 < \lambda < \frac{c - \sqrt{c^2 - 4dr \min\{b,1\}}}{2d}. 
$$ 
\end{remark}

\begin{figure}
\centering
\includegraphics[width=0.9\textwidth]{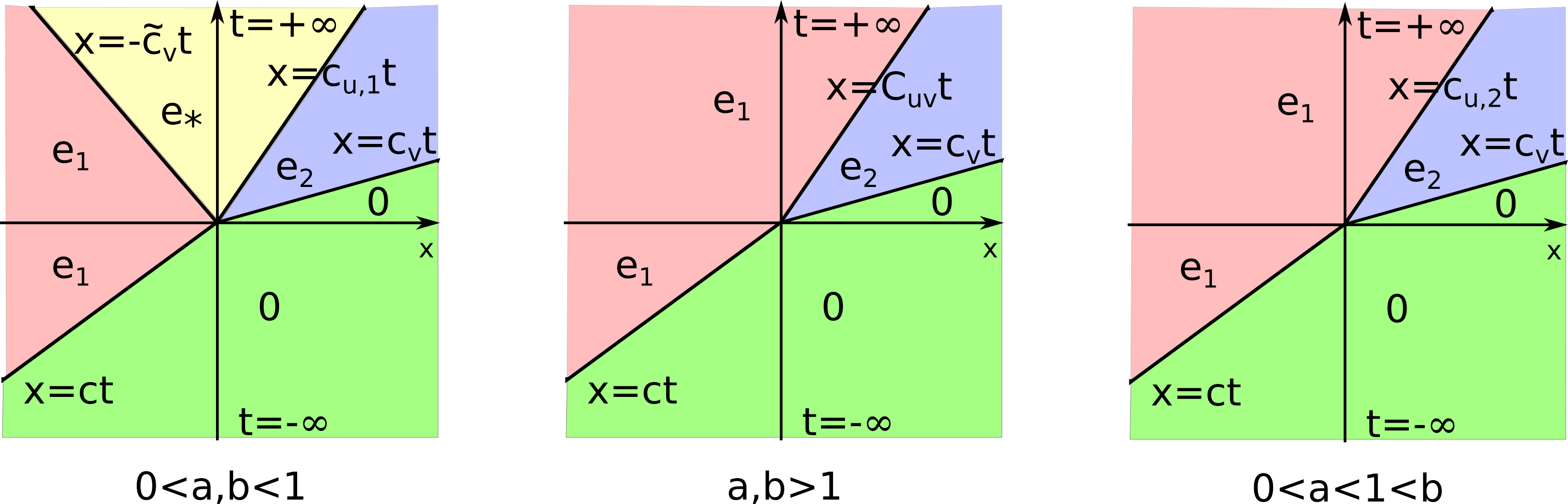}
\caption{Schematic plots of entire solutions in Theorem \ref{Main Tm}.}
\end{figure}


\medskip

By Theorem \ref{Main Tm} (1), one observes that for each $c\geq 2$ and $\lambda \in (0, \sqrt{r/d})$, the entire solution ${\bf u}_\lambda$ is approximately equal to ${\bf e}_*$ in the region
$$
\Omega_* =  \{(t,x):  -\widetilde{c}_v t < x < c_{u,1} t, \,\, \text{ and }\,\, t \gg 1\}. 
$$
It is worth pointing out that, both $\widetilde{c}_v$ and $c_{u,1}$ are increasing in terms of $\lambda$. i.e. $\Omega_*$ is increasing in $\lambda$. The following can be viewed as a limiting case of Theorem \ref{Main Tm}(1),  when $\widetilde{c}_v = \infty$. This 
happens whenever $c \geq 2\max\{1, \sqrt{drb}\}$ and $\lambda = (c - \sqrt{c^2 - 4drb})/(2d)$. 
\begin{theorem}[Limiting divergent type]\label{MT2}
Let $d, r, >0$, $0<a , b<1$, and $c > 2\max\{1, \sqrt{drb}\}$ be given, 
then the Lotka-Volterra system \eqref{main system} admits the traveling wave solution 
\[
\bm{\Phi}_c(t,x)=(\Phi_c(x-ct), 0),
\]
from which ``originates" an entire solution ${\bf u}(t,x):=(u(t,x),v(t,x))$, $t,x\in\R$; that is,
\begin{equation}\label{MT2.1}
\lim_{t\to-\infty}\|{\bf u}(t,\cdot)-\bm{\Phi}_c(t,\cdot)\|_{\infty}=0.
\end{equation}
Moreover, there exists $h_0 \in \mathbb{R}$ such that the ``destiny"--long time dynamics as $t\to+\infty$-- of this entire solution satisfies the following properties.
\begin{subequations}
\begin{align}
\lim_{t\to\infty}\sup_{x\leq(c_{u,1}-\varepsilon)t}|{\bf u}(t,x)-{\bf e}_*|_1=0,\quad \forall 0<\varepsilon\ll 1,\label{MT2.2}\\
\lim_{t \to \infty} \sup_{x> (c_{u,1} + \varepsilon)t} |{\bf u}(t,x) - \bm{\Psi}_{c_{v,3}}(x- c_{v,3}t - h_0) |_1 = 0,\quad \forall\ \varepsilon>0, \label{MT2.4}
\end{align}
\end{subequations}
 where 
 \begin{equation}\label{new-eqz1}
 c_{v,3}:=d\lambda_3+r\lambda_3^{-1} \quad  \text{ with } \quad \lambda_3 = \frac{c - \sqrt{c^2 - 4drb}}{2d} \in \left(0, \sqrt{\frac{r}{d}}\right), 
 \end{equation} and that $c_{u,1}=\max\{C_1^*,\widehat{c}_{acc}\}$.
\end{theorem}

\medskip

\begin{theorem}[Merging type]\label{MT-3}
Given $d>0$, $r>0$, $0<a<1<b$ and $c_v> 2\max\{\sqrt{rd},\sqrt{a}\}$, 
%
then the Lotka-Volterra system \eqref{main system} admits an entire solution ${\bf u}_m(t,x):=(u_m(t,x),v_m(t,x))$ connecting the following two traveling wave solutions 
\[
\bm{\Psi}_{c_v}(x-c_vt)=(0,\Psi_{c_v}(x-c_vt)) \quad \text{ and }\quad \bm{\Phi}_{c_{u,3}}(x-c_{u,3}t)=(\Phi_{c_{u,3}}(x-c_{u,3}t),0)
\]
that is, there exists $h_0 \in \mathbb{R}$ such that 
\begin{subequations}\label{MT3.1}
\begin{align}
\lim_{t\to-\infty} \sup_{ x\in\mathbb{R}} |{\bf u}_m(t,x)-\bm{\Psi}_{c_v}(x - c_v t)|_{1}=0,  \label{MT3.1.1}\\
\lim_{t \to \infty} \sup_{x \in \mathbb{R}} |{\bf u}_m(t,x) - \bm{\Phi}_{c_{u,3}}(x- c_{u,3}t - h_0) |_1 = 0,\label{MT3.1.2}
\end{align}
\end{subequations}
 where 
 \begin{equation}\label{new-eqz2}
  c_{u,3}: = \lambda_4+ \lambda_4^{-1}, \quad \text {with } \lambda_4:=\frac{1}{2}\left(c_v-\sqrt{c_v^2-4a} \right).
 \end{equation}
 We note that $\lambda_4<1$ due to the fact that $c_v>2\sqrt a$. In addition, $c_{u,3} - c_v  = \frac{1-a}{2a} (c_v + \sqrt{c_v^2 - 4a}) >0$.
\end{theorem}
\begin{remark}
Traveling wave solutions $\bm{\Psi}_{c_v}(x-c_v t)$ and ${\bf \Phi}(x- c_{u,3}t)$  can be viewed as equilibria in moving frames with distinctive speeds $c_v$ and $c_{u,3}$ respectively. Given that, the above entire solution can be regarded as a ``generalized" heteroclinic orbit connecting these two equilibria ${\bf\Psi}$ and ${\bf\Phi}$.
\end{remark}

\begin{figure}[h]
\centering
\begin{subfigure}{.5\textwidth}
  \centering
  \includegraphics[width=.6\linewidth]{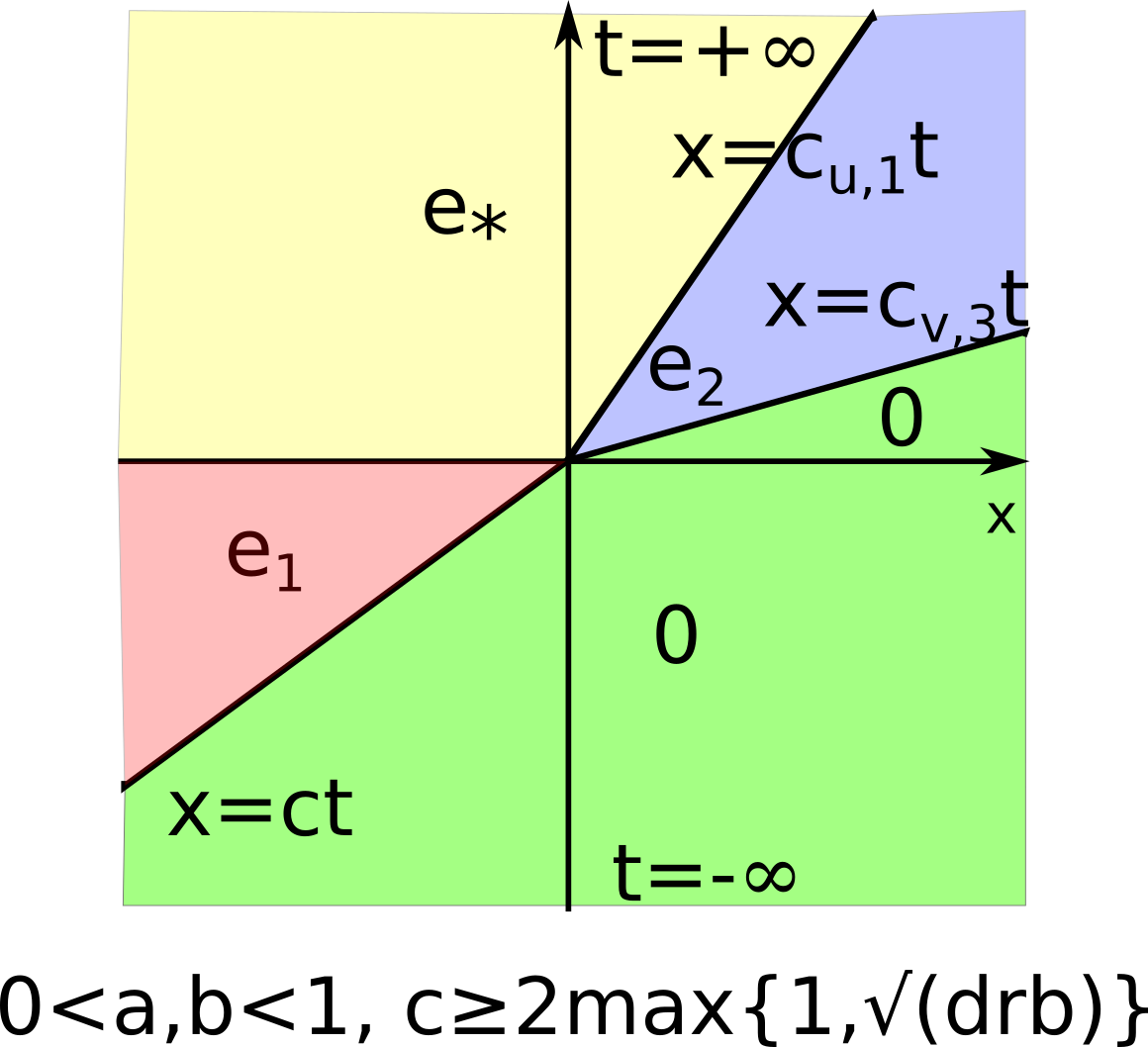}
  \label{fig:sub1}
\end{subfigure}%
\begin{subfigure}{.5\textwidth}
  \centering
  \includegraphics[width=.63\linewidth]{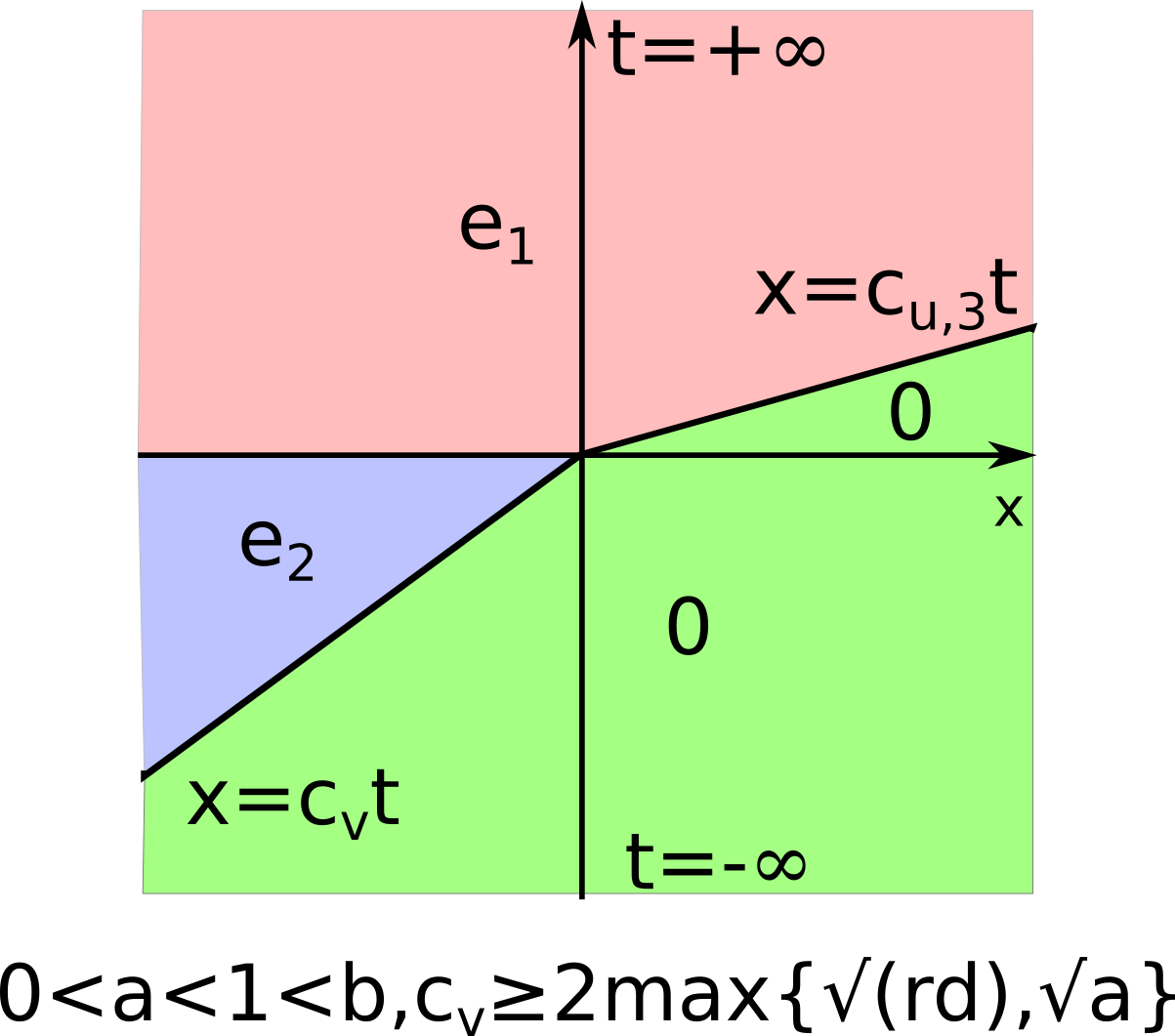}
  \label{fig:sub1}
\end{subfigure}%
\caption{Schematic plots of entire solutions in Theorem \ref{MT2} (Left) and Theorem \ref{MT-3}(Right).}
\label{fig:test}
\end{figure}


\medskip

The rest of the paper is organized as follows. In section \ref{s:2}, we study the eigenvalue problem associated to the linearized system of \eqref{main system} at $(\Phi_c,0)$. We then exploit the results from Section \ref{s:2} to establish the existence of entire solutions in section \ref{s:3}. The asymptotic behavior of diverging-type entire solutions are presented in section \ref{s:4}. The proof of Theorem \ref{MT-3} and Theorem \ref{MT2} are respectively presented in section \ref{s:5} and Section \ref{s:6}.

\section{Study of an  eigenvalue-problem }\label{s:2}

This section is devoted to the study of an eigenvalue problem of \eqref{main system} linearized at $(\Phi_c,0)$. The result of this section will be useful in the subsequent sections to construct a pair of super-solution and sub-solution of \eqref{main system}. Next, we show that \eqref{main system} has a unique entire solution  sandwiched between these super-solution and sub-solution. Introducing new notations
\[
\mu := g_{d,c,r}(\lambda), \quad \delta_v:=\frac{1}{2d}\left[\sqrt{c^2+4d(\mu+r(b-1))} -c\right],
\]
 our main results of this section read as follow.

\medskip

\begin{lemma}\label{existence-of-unstable-eigenfunction}
For each $c\ge 2 $ and and $0<\lambda<\min\{\sqrt{\frac{r}{d}}, \frac{c}{2d}\}$ such that  $g_{d,c,r}(\lambda) \geq r\max\{0, 1-b\}$.  There
exists a unique solution $\bm{\Phi}_e:=(\varphi,\psi) \in C^2(\mathbb{R})$ to
\begin{equation}\label{eigenv-eq}
\begin{cases}
\mu\varphi=\varphi_{xx}+c\varphi_{x}+(1-2\Phi_{c}(x))\varphi -a\Phi_{c}\psi,\quad x\in\R,\cr
\mu\psi=d\psi_{xx}+c\psi_{x}+r(1-b\Phi_{c}(x))\psi,\quad x\in\R,\cr
\varphi(\pm\infty)=(0,0),\quad \varphi<0<\psi,\cr
{\displaystyle \limsup_{x \to -\infty } e^{-\delta_v x} \psi(x) < +\infty, \quad \text{ and }\quad 
\lim_{x \to \infty} e^{\lambda x} \psi(x) = 1. }
\end{cases}
\end{equation}
Moreover,  there is a positive constant $\Upsilon$ such that 
\begin{equation}\label{Eq:exact-decay-of-psi-at-infty}
\lim_{x\to-\infty}e^{-\delta_v x}\psi(x)=\Upsilon,
\end{equation}
where $\delta_v \geq 0$.

\end{lemma}

\begin{lemma}\label{lem:psi}
 Given $c \geq 2$ and $\lambda \in \left(0, \frac{c}{2d}\right)$, we have the following results. 
\begin{itemize}
\item[(a)] There exists a unique solution $\psi \in C^2(\mathbb{R})$ to 
\begin{equation}\label{eigenv-eq1}
\begin{cases}
\mu \psi = d \psi_{xx} + c \psi_x + r(1-b \Phi_c(x)) \psi, \quad \text{ for }\,\,x \in \mathbb{R},\\
\psi >0, \quad  \qquad \qquad \qquad \qquad \qquad \qquad  \text{ for }\,\,x \in \mathbb{R},\\
{\displaystyle \limsup_{x \to -\infty } e^{-\delta_v x} \psi(x) < +\infty, \quad \text{ and }\quad 
\lim_{x \to \infty} e^{\lambda x} \psi(x) = 1. }
\end{cases}
\end{equation}
\item[(b)] There exists $D>0$ such that  
$$e^{-\lambda x} - D e^{-\widetilde\lambda x} \leq \psi(x) \leq e^{-\lambda x}, \quad \text{ for }x \gg 1.
$$
\item[(c)] The function $x\mapsto e^{-\delta_vx}\Psi(x)$ is decreasing, and there exists $\Upsilon>0$ such that
$$
\lim_{x \to -\infty} e^{-\delta_v x} \psi(x) = \Upsilon. 
$$
\item[(d)] In particular, if $g_{d,c,r}(\lambda) = r(1-b)$, then  $\delta_v = 0$, the function $\psi(x)$ is decreasing and there exists $\Upsilon>0$ such that
$$
\lim_{x \to -\infty} \psi(x)  = \Upsilon.
$$

\end{itemize}

\end{lemma}
\begin{proof}
First, we prove the uniqueness part of (a). Let $\psi_1(x), \psi_2(x)$ be two solutions to \eqref{eigenv-eq1}. Let $\widetilde{\psi}(x)=\frac{\psi_1(x)}{\psi_2(x)}$. We shall show that $\widetilde{\psi}(x)\equiv 1$.  
 By setting $h(x)=e^{\frac{c}{d}x}$ and $k(x)=\frac{h(x)}{d}\left(r(1-b\Phi_c(x))-\mu\right)$, both $\psi_1(x)$ and $\psi_2(x)$ satisfies 
$$ 
\left( h(x)\psi'(x) \right)'+k(x)\psi(x)=0.
$$
 By Lagrange identity, it holds that for $i\neq j$, $i,j\in\{1,2\}$, 
$$ 
\left(h(x)\left(\psi_i'\psi_j-\psi_j'\psi_i\right)\right)'=0.
$$
As a result, there is a constant  $c_{ij}\in\mathbb{R}$ such that 
$$ 
h(x)\left(\psi_i'\psi_j-\psi_j'\psi_i\right)=c_{ij},\quad \forall x\in\mathbb{R},
$$
equivalently, 
$$ 
\left(\frac{\psi_i}{\psi_j}\right)'(x)=\frac{c_{ij}}{h(x)[\psi_j(x)]^2},\quad \forall x\in\mathbb{R}.
$$
 Integrating both sides yields
$$ 
\left(\frac{\psi_i}{\psi_j}\right)(y) =  \left(\frac{\psi_i}{\psi_j}\right)(x) + \int_{x}^{y}\frac{c_{ij}}{h(s)[\psi_j(s)]^2}ds,\quad \forall \ y\leq x\in\mathbb{R}.
$$
Letting $y\to\infty$ in this equation and exploiting that $\displaystyle \lim_{x \to +\infty} e^{\lambda x}\psi_i(x) = 1$ for $i=1,2$,  we obtain 
\begin{equation}\label{eq:wwww1} 
1=  \left(\frac{\psi_i}{\psi_j}\right)(x) + c_{ij}\int_{x}^{\infty}\frac{1}{h(s)[\psi_j(s)]^2}ds,\quad \forall y\in\mathbb{R},
\end{equation}
 which, due to the fact that $\psi_1(x), \psi_2(x)>0$ for every $x\in\mathbb{R},$ yields
\begin{equation}\label{eq:wwww2}
1\geq c_{ij}\int_{x}^{\infty}\frac{1}{h(s)[\psi_j(s)]^2}ds,\quad \forall x\in\mathbb{R}.
\end{equation}
Observe, however, that for $s \to -\infty$,   
$$
\limsup_{s \to -\infty} h(s)[\psi_j(s)]^2 = \limsup_{s \to -\infty}   \exp((\frac{c}{d} + 2\delta_v)s) [\exp(-\delta_v s) \psi_j(s)]^2 \leq 1, 
$$ since $\delta_v  \geq-\frac{c}{2d}$. Hence, we deduce that  $$ \int_{-\infty}^{\infty}\frac{1}{h(s)[\psi_j(s)]^2}ds=\infty,$$ 
which, together with \eqref{eq:wwww2}, shows that 
\begin{equation}\label{eq:c}
c_{ij}\leq 0.
\end{equation}
Combining \eqref{eq:wwww1} and \eqref{eq:c}, we deduce that  $\psi_i(x) \geq \psi_j(x)$ for $x \in \mathbb{R}$.
Since $i\neq j$ are arbitrary chosen in $\{1,2\}$, we conclude that $\psi_1(x) = \psi_2(x)$ for every $x\in\mathbb{R}$, which proves the uniqueness part of (a).

For existence, we now construct a pair of super- and sub-solutions. First, define
$$
\underline\psi_1(x;\lambda,\widetilde\lambda,D):= \max\{0,e^{-\lambda x} -D e^{-\widetilde\lambda x}\} = \begin{cases} e^{-\lambda x} -D e^{-\widetilde\lambda x} &\text{ for }x > (\widetilde\lambda - \lambda)^{-1} \log D,\\
0 &\text{ for }x \leq  (\widetilde\lambda - \lambda)^{-1}\log D, \end{cases}
$$
where $\widetilde\lambda \in (\lambda, \lambda+\tau_c)$ is chosen close enough to $\lambda$ so that $g_{d,c,r}(\widetilde\lambda) < g_{d,c,r}(\lambda)$, thanks to the fact that $g'_{d,c,r}(\lambda) = 2d\lambda -c <0$. Recall that $\tau_c=\frac{1}{2}(c-\sqrt{c^2-4})\geq0$.  We claim that $\underline\psi_{1}$ is a weak sub-solution of \eqref{eigenv-eq1} for $D \gg 1$. Indeed, introducing the notations $\mathcal{L}:=d\partial_x^2+c\partial_x+r(1-b\Phi_c)$ and $\underline\psi_{\lambda,\widetilde\lambda,D}:=e^{-\lambda x}-De^{-\widetilde\lambda x}$, we have
\[
\begin{aligned}
 - \mathcal{L}(\underline\psi_{\lambda,\widetilde\lambda,D}) + \mu \underline\psi_{\lambda,\widetilde\lambda,D}&= rb\Phi_c(x)e^{-\lambda x}+D\left[ (g_{d,c,r}(\widetilde{\lambda})-g_{d,c,r}(\lambda)) -rb\Phi_c(x) \right]e^{-\widetilde\lambda x}\cr
&{\overset{\text{\eqref{decay-estimate for phi_c}}}{\leq}}  rb e^{-\tau_c x}e^{-\lambda x}-D(g_{d,c,r}(\lambda)-g_{d,c,r}(\widetilde{\lambda}))e^{-\widetilde\lambda x}-Drb\Phi_{c}(x)e^{-\widetilde\lambda x}\cr
&\leq rbe^{-(\lambda+\tau_c)x}-D(g_{d,c,r}(\lambda)-g_{d,c,r}(\widetilde{\lambda}))e^{-\widetilde\lambda x}\cr 
&= \left( rb e^{-(\tau_c+\lambda-\widetilde{\lambda})x}-D(g_{d,c,r}(\lambda)-g_{d,c,r}(\widetilde{\lambda}))\right)e^{-\widetilde{\lambda}x} \cr
&\leq\left( rb  -D(g_{d,c,r}(\lambda)-g_{d,c,r}(\widetilde{\lambda}))\right)e^{-\widetilde{\lambda}x} \leq  0,
\end{aligned}
\]
provided $D  \geq (rb)/(g_{d,c,r}(\lambda)-g_{d,c,r}(\widetilde{\lambda})) $.


Next, we construct a super-solution $\overline\psi_1$. Let $\varepsilon_2 \in (0, \widetilde{\tau}_c)$ where $\widetilde{\tau}_c=\frac{1}{2}(\sqrt{c^2+4}-c)$ and  $x_2 \ll -1$ in the sense of $x_2<0$ and $|x_2|$ sufficiently large. We define
$$
\overline\psi_1(x; \delta_v,\varepsilon_2, \lambda, x_2):= \begin{cases} K_0 e^{\delta_v x}(1 -e^{\varepsilon_2  x}) &\text{ for }x \leq  x_2,\\
e^{-\lambda x} &\text{ for }x > x_2, \end{cases}
$$
where $K_0 := \frac{e^{-\lambda x_2}}{e^{\delta_v x_2}(1 -e^{\varepsilon_2  x_2})}$. Since $e^{-\lambda x}$ is obviously a super-solution in $\mathbb{R}$, it remains to show that $e^{\delta_v x}(1-e^{  \varepsilon_2  x})$, and thus $K_0e^{\delta_v x}(1-e^{\varepsilon_2  x})$, is a super-solution for $x \ll -1$. Indeed,  noting that 
\[
\begin{aligned}
d\delta_v^2+c\delta_v+r(1-b)-\mu&=0,\\
-\mathcal{L}e^{\delta_v x}+\mu e^{\delta_v x}+rb(1-\Phi_c)e^{\delta_vx}&=0,
\end{aligned}
\]
 we have, for $x\ll -1$,
\[
\begin{aligned}
&\quad -\mathcal{L} (e^{\delta_v x}(1-e^{  \varepsilon_2  x})) +\mu (e^{\delta_v x}(1-e^{  \varepsilon_2  x}))\\
 &=   e^{(\delta_v + \varepsilon_2) x} \left[d(\delta_v + \varepsilon_2)^2 + c (\delta_v + \varepsilon_2) + (r(1-b) - \mu) -rb(1-\Phi_c(x))e^{-\varepsilon_2 x}(1-e^{\varepsilon_2 x})\right]  \\
 &=   e^{(\delta_v + \varepsilon_2) x} \left[d\varepsilon_2(2\delta_v + \varepsilon_2) + c \varepsilon_2  -rb(1-\Phi_c(x))e^{-\varepsilon_2 x}(1-e^{\varepsilon_2 x})\right]  \\
&\overset{\eqref{decay-estimate for phi_c-at-infty}}{\geq}e^{(\delta_v + \varepsilon_2) x} \big[d\varepsilon_2(2\delta_v +\frac{c}{d}+ \varepsilon_2)   -O(e^{(\widetilde{\tau}_c-\varepsilon_2)x})\big]   \\
%
&>0,
\end{aligned}
\]
where the last inequality follows from the facts that $\delta_v\geq-c/{2d}$ and $\widetilde{\tau}_c - \varepsilon_2 >0$, so that the term in the square bracket is positive for $x\ll -1$. Hence, we have proved that $\overline\psi_1$ is a super-solution of \eqref{eigenv-eq1}. Now, fix $\varepsilon_2 \in (0,\widetilde{\tau}_c)$, then  $\underline\psi_1 (x) < \overline\psi_1(x)$ in $\mathbb{R}$ provided $D \gg 1$ and $x_2 \ll -1$. It follows from standard method of super- and sub-solutions that \eqref{eigenv-eq1} has a solution $\psi$ satisfying $\underline\psi_1(x) \leq \psi(x) \leq \overline\psi_1(x)$ in $\mathbb{R}$. This proves (a) and (b). We observe in addition that 
\begin{equation}\label{eq:psi121}
\psi(x) \leq \overline\psi_1 \leq  K_0e^{\delta_v x}  \quad \text{ for }x \ll -1.
\end{equation}

\medskip
 Next, we prove that (c) holds. Indeed, let $\tilde{\delta}_v$ denotes the negative root of 
$$ 
0=d\delta^2+c\delta + r(1-b)+\mu.
$$
Using the fact that $d\delta_v\tilde{\delta}_v=r(1-b)-\mu $ and $ d(\delta+\tilde{\delta}_v)=-c$, we obtain
$$ 
\frac{d}{dx}\left((\psi'(x)-\delta_v\psi(x))e^{-\tilde{\delta}_vx} \right)= \frac{e^{-\tilde{\delta}_vx}}{d}\left(d\psi''(x)+c\psi'(x)+(r(1-b)-\mu)\psi(x)\right)=\frac{rb}{d}\psi(x)(\Phi_c(x)-1)e^{-\tilde{\delta}_vx}<0.
$$
That is the function $x\mapsto(\psi'(x)-\delta_v\psi(x))e^{-\tilde{\delta}_vx}$ is strictly decreasing. Since $\psi\in C^b_{\rm unif}(\mathbb{R})$, standard Hanack's inequalities for elliptic equations imply that $\psi'\in  C^b_{\rm unif}(\mathbb{R})$, and hence $\lim_{x\to-\infty}(\psi'(x)-\delta_v\psi(x))e^{-\tilde{\delta}_vx}=0$. Thus $(\psi'(x)-\delta_v\psi(x))e^{-\tilde{\delta}_vx}<0$ for every $x\in\mathbb{R}$. Which implies that $\frac{d}{dx}(e^{-\delta_v x}\psi(x))<0$ for every $x\in\mathbb{R}$. This together with \eqref{eq:psi121} complete the proof of (c).

Finally, since $(d)$ follows from $(c)$, the proof of the lemma is complete.

\end{proof}

\begin{remark}
If, in addition, we assume that $\delta_v>0$, and given  $x_2 \ll -1$, we have $K_0(1-e^{\varepsilon_2  x})<e^{-(\lambda + \delta_v) x}$ for any $x\leq x_2$, yielding 
$$
{\overline\psi_1}(x) = e^{\delta_v x} \min\{K_0(1-e^{\varepsilon_2  x}), e^{-(\lambda + \delta_v) x}\} = \min\{K_0e^{-\delta_v x}(1-e^{ - \varepsilon_2  x}), e^{-\lambda x}\}, \quad \forall x\leq x_2.
$$
\end{remark}
\begin{remark}We can prove a more general result using dynamical systems and functional analysis argument; see the appendix for details.\end{remark}

Next we present the  proof of Lemma \ref{existence-of-unstable-eigenfunction}. 

\medskip

\begin{proof}[Proof of Lemma \ref{existence-of-unstable-eigenfunction} {\rm(i)}] Fix $\lambda \in \Lambda_{d,c,r,b}$, and let $\psi(x)$ be given by Lemma \ref{lem:psi}. It follows from  \eqref{eigenv-eq1} 
that $\psi(x)$ is a positive eigenfunction corresponding to $\mu=g_{d,c,r}(\lambda)$ for the linear operator $\mathcal{L}$ arising from the second equation of the elliptic system \eqref{eigenv-eq}. Moreover, Lemma \ref{lem:psi} (a)-(c) say that $\psi(x)$ satisfies the desired asymptotic behaviors at $x=\pm\infty$, including  \eqref{Eq:exact-decay-of-psi-at-infty},  as stated in Lemma \ref{existence-of-unstable-eigenfunction}.  Since the uniqueness of $\psi$ has also been proved in Lemma \ref{lem:psi}(a), it remains to determine $\varphi$ by solving the first equation of \eqref{eigenv-eq}. 

Note that  $\varphi$ solves the first equation in \eqref{eigenv-eq} if and only if the function $\phi=\frac{\varphi}{\Phi_{c}}$ satisfies 
\begin{equation}\label{a-2}
\mu \phi=\phi'' +\left(2\frac{\Phi_c'}{\Phi_c}+c\right)\phi' -\Phi_{c}\phi - a\psi.
\end{equation} 

Let $C_{0}(\R)$ denotes the Banach space 
$$ 
C_{0}(\R):=\{ u \in C(\R)\ \mid \lim\limits_{x\to\pm\infty}\ u(x)=0 \} \quad 
$$
endowed with the sup-norm $\|u\|_{C_{0}(\R)}:=\|u\|_{\infty}$. Note that, since $\Phi_c(x)>0$ for every $x\in\R$, the linear operator
$$ \mathcal{L}_{\Phi_c} (\phi):=\phi'' +\left(2\frac{\Phi_c'}{\Phi_c}+c\right)\phi' -\Phi_{c}\phi$$
generates an analytic semigroup of contractions on $C_{0}(\R)$. Hence, the Hille-Yosida Theorem implies that for every $\mu>0$ (and $\mu = g_{d,c,r}(\lambda)$ in particular), one can solve \eqref{a-2} for a unique solution $\widetilde{\phi}\in C_{0}(\R)$. 
Moreover, since $-a\psi(x)<0$ for every $x\in\R$, the maximum principle implies that $\widetilde{\phi}(x)<0$ for every $x\in\R$. Therefore, taking $\varphi=\widetilde{\phi}\Phi_c$, it holds that $(\psi,\varphi)$ solves \eqref{eigenv-eq}.
\end{proof}

\begin{remark}\label{rk1} We note from the proof of Lemma \ref{existence-of-unstable-eigenfunction} that $\frac{\varphi}{\Phi_c}\in C_0(\mathbb{R})$, that is,
$$ 
\lim_{|x|\to\infty}\frac{\varphi(x)}{\Phi_c(x)}=0\quad \text{and}\quad \frac{\varphi}{\Phi_c}\in C^b_{\rm unif}(\R).
$$
\end{remark}

\section{Existence of entire solutions} \label{s:3}
In this section we construct entire solutions of \eqref{main system}. Thanks to Lemma \ref{existence-of-unstable-eigenfunction}, we are able to construct a pair of super-solutions and sub-solution of \eqref{main system} which implies the existence of 
a unique entire solution sandwiched between them. The asymptotic behavior of these entire solution at $t=-\infty$ can then be inferred from the behaviors of the pair of super-sub-solutions.

\subsection{Existence of entire solutions of Theorem \ref{Main Tm}.}\label{sec-3} 
Through this subsection we fix $c\geq 2$, $\lambda\in\Lambda_{d,c,r,b}$, let $g_{d,c,r}(\lambda)=\mu$ and $(\varphi,\psi)$ be the solution of \eqref{eigenv-eq} given by Lemma \ref{existence-of-unstable-eigenfunction}. We introduce the co-moving frame $\xi=x-ct$ and rewrite \eqref{main system} as 
\begin{equation}\label{rewrite-main-system} 
{\bf u}_t=\bm{\mathcal{A}}_c({\bf u})
\end{equation}
where ${\bf u}=(u,v)$ and 
\[\bm{\mathcal{A}}_c({\bf u})=(\mathcal{A}_{1,c}({\bf u}),\mathcal{A}_{2,c}({\bf u})):=(u_{\xi\xi}+cu_\xi+u(1-u-av), dv_{\xi\xi}+cv_\xi+rv(1-v-bu)).\]
We note that $(u(t,\xi),v(t,\xi))$ is an entire solution of \eqref{rewrite-main-system} if and only if $(u(t,x-ct),v(t,x-ct))$ is entire solution of \eqref{main system}. Hence in the following we only need to prove the existence of entire solution of \eqref{rewrite-main-system}.

For the convenience of stating the main results of this section, we first introduce the following lemma.
 \begin{lemma}\label{asymptotic-behavior-of-p-and-q}  Given $M>0$ and $0<\varepsilon<\frac{\mu}{M}$, 
both components of the solution, $p(t)$ and $q(t)$, to the system
\begin{equation}\label{p(t)-q(t)-ODE}
\begin{cases}
\dot{p}=\mu+\varepsilon Me^{p(t)},\quad p(0)= - \log\left(1-\frac{\varepsilon M}{\mu}\right),\cr
\dot{q}=\mu-\varepsilon M e^{q(t)}.\quad q(0)=-\log\left(1+\frac{\varepsilon M}{\mu}\right).
\end{cases}
\end{equation}
 are increasing functions which satisfy 
 \begin{subequations}\label{eq:pqprop}
 \begin{align}
 \lim_{t\to-\infty}|p(t)-\mu t|=\lim_{t\to-\infty}|q(t)-\mu t|=0, \label{eq:pqprop1}\\
 p(t)\geq q(t),\quad \forall\ t\leq 0, \label{eq:pqprop3}\\
 \lim_{t\to+\infty} e^{q(t)}=\frac{\mu}{\varepsilon M}. \label{eq:pqprop2}
 \end{align}
 \end{subequations}
 \end{lemma}
\begin{proof}
Solve explicitly, we have 
\begin{equation}\label{p-formula}
p(t)=\mu t -\log\left(1-\frac{\varepsilon M\text{exp}(\mu t)}{\mu}\right) = -\log\left(- \frac{\varepsilon M}{\mu} +\exp(-\mu t) \right),  \quad \text{ for }t \leq 0,
\end{equation}
and 
\begin{equation}\label{q-formula}
q(t)=\mu t -\log\left(1+\frac{\varepsilon M\text{exp}(\mu t)}{\mu}\right) = -\log\left(\frac{\varepsilon M}{\mu} +\exp(-\mu t) \right),\quad t\in\R. 
\end{equation}
It follows that $p(t)$ and $q(t)$ are increasing and satisfies \eqref{eq:pqprop}
\end{proof}
%
We remark that the functions $p(t)$ and $q(t)$ have also been used in \cite{FuMoNi} to prove similar results for the Allen-Cahn equation to our main results here.  We also introduce the following definitions.
%
\begin{definition}  Let ${\bf u}_1=(u_1, v_1), {\bf u}_0=(u_0, v_0)\in\R^2$, and ${\bf u}(t,\xi)=(u(t,\xi),v(t,\xi))$ be a piecewise smooth function on $I\times\R$, where $I\subseteq\R$ is an open interval. 
\begin{itemize}
\item[(i)] We say that ${\bf u}_0\le_{K}{\bf u}_1$ if  $u_0\leq u_1$ and $v_0\geq v_1$.
\item[(ii)] The function ${\bf u}(t,\xi)=(u(t,\xi),v(t,\xi))$ is a sub-solution of \eqref{rewrite-main-system} on $I\times\R$ if 
$$ 
{\bf u}_t\le_{K} \bm{\mathcal{A}}_c({\bf u}), \quad \text{ in the weak sense for }(t,\xi) \in I \times \mathbb{R}.
$$
\item[(iii)] The function ${\bf u}(t,\xi)=(u(t,\xi),v(t,\xi))$ is a super-solution of \eqref{rewrite-main-system} on $I\times\R$ if 
$$ 
\bm{\mathcal{A}}_c({\bf u})\le_{K} {\bf u}_t, \quad \text{ in the weak sense for }(t,x) \in I \times \mathbb{R}.
$$
\end{itemize}
For more precise definition of weak super-sub-solutions, we refer to \cite[Sect. 2.1]{LeoLam2018}.
\end{definition}

The following result is well known.
\begin{prop}[Comparison principle for \eqref{rewrite-main-system}]\label{Comparison-principle}
Suppose that 
\[
\underline{\bm{\Phi}}(t,\xi):=(\underline{u}(t,\xi),\overline{v}(t,\xi)), \quad \overline{\bm{\Phi}}(t,\xi):=(\overline{u}(t,\xi),\underline{v}(t,\xi))\in C([t_0,t_0+T)\times\R)\cap C^{1,2}((t_0,t_0+T)\times\R)
\]
are respectively sub-solution and super-solution of \eqref{rewrite-main-system} on $(t_0,t_0+T)\times\R$. If $\underline{\bm{\Phi}}(t_0,\xi)\leq_K \overline{\bm{\Phi}}(t_0,\xi)$ for every $\xi\in\R$, then
$$ 
\underline{\bm{\Phi}}(t_0+t,\xi))\leq_K \overline{\bm{\Phi}}(t_0+t,\xi)),\quad \forall\ 0< t<T, \ \xi\in\R.
$$
\end{prop}

We now set 
\begin{subequations}
\begin{align}
\underline{\bm{\Phi}}_\star(t,\xi)&=(\underline{u}_\star(t,\xi),\overline{v}_\star(t,\xi)):=\bm{\Phi}_c(\xi)+\varepsilon e^{p(t)}\bm{\Phi}_e(\xi),\quad t\leq 0, \ \xi\in\R, \label{super-sol-of-main-system}\\
\overline{\bm{\Phi}}_\star(t,\xi)&=(\overline{u}_\star(t,\xi),\underline{v}_\star(t,\xi)):=\bm{\Phi}_c(\xi)+\varepsilon e^{q(t)}\bm{\Phi}_e(\xi), \quad (t,\xi)\in\R\times\R, \label{sub-sol-of-main-system}
\end{align}
\end{subequations}
where $\bm{\Phi}_e=(\varphi,\psi)$ is the solution of \eqref{eigenv-eq} given by Lemma \ref{existence-of-unstable-eigenfunction} (i),  $p(t)$ and $q(t)$ are given by Lemma \ref{asymptotic-behavior-of-p-and-q},
and state our main result in this section. 

\begin{theorem}\label{existence-of-entire-solution}
Let $c\geq 2$ and $\lambda\in(0,\sqrt{r/d})$ be given such that $g_{d,c,r}(\lambda)>0$ and $g_{d,c,r}(\lambda) \geq r(1-b)$. Let $M> \max\{\|\varphi+a\psi\|_{\infty},r\|b\varphi+\psi\|_{\infty}\}$ and $0<\varepsilon<\frac{\mu}{M}$.  There is a unique entire solution $\bm{\Phi}_*(t,\xi):=(u_*(t,\xi),v_*(t,\xi))$ of \eqref{rewrite-main-system} satisfying for $(t,\xi)\in(-\infty, 0]\times\R$ that 
\begin{equation}\label{asymptotic at -infty}
\underline{\bm{\Phi}}_\star(t,\xi)\leq_K \bm{\Phi}_*(t,\xi)\leq_K \overline{\bm{\Phi}}_\star(t,\xi).
\end{equation}
\end{theorem}
Equivalently, we have the following for \eqref{main system}.
\begin{corollary}\label{coro:mai-existence-result}
Let $c\geq 2$ and $\lambda\in(0,\sqrt{r/d})$ be given such that $g_{d,c,r}(\lambda)>0$ and $g_{d,c,r}(\lambda) \geq r(1-b)$. Let $M> \max\{\|\varphi+a\psi\|_{\infty},r\|b\varphi+\psi\|_{\infty}\}$ and $0<\varepsilon<\frac{\mu}{M}$. There is a unique entire solution 
\[
{\bf u}_\lambda(t,x)=(u_\lambda(t,x),v_\lambda(t,x)):=(u_*(t,x-ct),v_*(t,x-ct))
\]
of \eqref{main system} satisfying  for $(t,x)\in(-\infty, 0]\times\R$ that 

\begin{equation}\label{main-asymptotic at -infty}
\underline{\bm{\Phi}}_\star(t,x-ct)\leq_K {\bf u}_\lambda(t,x)\leq_K \overline{\bm{\Phi}}_\star(t,x-ct).
\end{equation}
\end{corollary} 
\begin{remark}\label{rmk:jj}
From \eqref{main-asymptotic at -infty} we can observe that ${\bf u}_\lambda(0,-\infty) =(1,0)$ and ${\bf u}_\lambda(0,\infty)=(0,0)$. 
Next, we utilize \eqref{decay-estimate for phi_c}  and Remark \ref{rk1} to derive that 
$$
\lim_{|x| \to\infty} \frac{u_\lambda(0,x)}{\Phi_c(x)} = 1, \quad \text{ i.e., }\quad \lim_{x \to -\infty} u_\lambda(0,x) = 1, \quad \text{ and }\quad \lim_{x \to +\infty}e^{\tau_c x} u_\lambda(0,x) = 1,
$$
where $\tau_c = \frac{1}{2}(c - \sqrt{c^2 - 4})$.
\end{remark}
\begin{remark}\label{rk2} We point out that the entire solution $\bm{\Phi}_*(t,\xi)$ of \eqref{rewrite-main-system} provided by Theorem \ref{existence-of-entire-solution} depends on $\varepsilon$ for each fixed $\varepsilon<\frac{\mu}{M}$ and its time translations form a class of entire solutions.
\end{remark}
\begin{remark} \label{rmk:bb}
Note by Remark \ref{rk1} that  we may choose $\varepsilon$ sufficient small so that for any $t\leq0$ and $\xi\in\R$,
$$ 
\underline{u}_\star(t,\xi)=\Phi_c(\xi)+\varepsilon\varphi(\xi)e^{p(t)}>0,\quad 
\overline{v}_\star(t,\xi)=\varepsilon\psi(\xi)e^{p(t)}< 1.
$$
\end{remark} 
%

\medskip 

\begin{lemma}\label{enrire-super-sub-sol}Let $c\geq 2$ and $\lambda\in(0,\sqrt{r/d})$ be given such that $g_{d,c,r}(\lambda)>0$ and $g_{d,c,r}(\lambda) \geq r(1-b)$. Let $M> \max\{\|\varphi+a\psi\|_{\infty},r\|b\varphi+\psi\|_{\infty}\}$ and $0<\varepsilon<\frac{\mu}{M}$. Then we have 
\begin{itemize}
\item[(i)] $\underline{\bm{\Phi}}_\star$ (resp. $\overline{\bm{\Phi}}_\star$ ) is a sub-solution (resp. super-solution) of \eqref{main system} on $(-\infty,0]\times\R$ (resp.  $\R\times\R$).
\item[(ii)] $\underline{\bm{\Phi}}_\star(t,\xi)\leq _K \overline{\bm{\Phi}}_\star(t,\xi)$ for every $(t,\xi)\in(-\infty,0]\times\R$.
\end{itemize}
\end{lemma}
\begin{proof} To prove (i), 
observe from \eqref{tw-eq} and Lemma \ref{existence-of-unstable-eigenfunction} 
\begin{align*} 
\mathcal{A}_{1,c}(\underline{u}_\star,\overline{v}_\star)=&\left[ \Phi_c'' +c\Phi_c' +\Phi_c\left(1-\Phi_c \right)\right] +\varepsilon e^{p(t)}\left[\varphi_{\xi\xi}+c\varphi_\xi+\varphi(1-2\Phi_c)-a\Phi_c\psi -\varepsilon\varphi\left(\varphi+a\psi \right)e^{p(t)} \right]\cr
=&\varepsilon \varphi e^{p(t)}\left[\mu -\varepsilon\left(\varphi+a\psi \right)e^{p(t)} \right],
\end{align*}
which, together with $\varphi(x)<0$ from \eqref{p(t)-q(t)-ODE}, yields
\begin{align*}
\partial_t\underline{u}_\star-\mathcal{A}_{1,c}(\underline{u}_\star,\overline{v}_\star)= \varepsilon \varphi e^{p(t)}\left[\dot{p}-\mu +\varepsilon\left(\varphi+a\psi \right)e^{p(t)} \right]
=\varepsilon^2 \varphi e^{2p(t)}\left[M +\left(\varphi+a\psi \right) \right]
\le0.
\end{align*}
Similarly, it also follows from Lemma \ref{existence-of-unstable-eigenfunction} that
\begin{align*} 
\mathcal{A}_{2,c}(\underline{u}_\star,\overline{v}_\star)=& \varepsilon e^{p(t)}\left[d\psi_{xx}+c\psi_x+r\psi(1-b\Phi_c)-r\varepsilon\psi\left(b\varphi+\psi \right)e^{p(t)} \right]\cr
=&\varepsilon \psi e^{p(t)}\left[\mu -r\varepsilon\left(b\varphi+\psi \right)e^{p(t)} \right],
\end{align*}
which, together with  $\varphi(x)<0$ from \eqref{p(t)-q(t)-ODE}, yields
\begin{align*}
\partial_t\overline{v}_\star-\mathcal{A}_{2,c}(\underline{u}_\star,\overline{v}_\star)=\varepsilon \psi e^{p(t)}\left[\dot{p}-\mu +\varepsilon r\left(b\varphi+\psi \right)e^{p(t)} \right]
=\varepsilon^2 \psi e^{2p(t)}\left[M +r\left(b\varphi+\psi \right) \right]
\ge0.
\end{align*}
As a result, $\underline{\bm{\Phi}}_\star$ is a sub-solution  of \eqref{main system} on $(-\infty,0]\times\R$. Similarly we can also show that $\overline{\bm{\Phi}}_\star$ is a super-solution  of \eqref{main system} on $\R^2$.

Finally, (ii) follows from Lemma \ref{asymptotic-behavior-of-p-and-q} along with the fact that $\varphi(x)<0<\psi(x)$ for every $x\in\R$.
\end{proof}

\begin{remark} Observe that 
$$ 
\|\varepsilon e^{q(0)} \psi\|_{\infty}=\frac{\varepsilon}{1+\frac{\varepsilon M}{\mu}}\|\psi\|_{\infty}\to 0 \quad \text{as}\ \varepsilon\to 0^+. 
$$

\end{remark}
For every ${\bf u}_0(\xi):=(u_0(\xi),v_0(\xi))\in C^b_{\rm unif}(\R)\times C^b_{\rm unif}(\mathbb{R})$ and $t_0\in\R$, let 
\[
{\bf u}(t, \xi; t_0, {\bf u}_0)=(u(t,\xi;t_0,{\bf u}_0),v(t,\xi;t_0,{\bf u}_0)), \quad t\geq t_0,\quad  \xi\in\R,
\]
denote the classical solution of 
\begin{equation*}
\begin{cases}
{\bf u}_t=\bm{\mathcal{A}}_c({\bf u}), \quad t>t_0,\ \xi\in\R,\cr
{\bf u}(t_0,\xi)={\bf u}_0(\xi),\ \xi\in\R.
\end{cases}
\end{equation*}
Throughout the rest of this work we fix $M$ and $\varepsilon$ such that the assumptions of Lemma \ref{enrire-super-sub-sol} are satisfied. 
For every $n\in\mathbb{Z}^+$, $\xi\in\R$ and $t\in[-n,0]$, we introduce 
\begin{equation*}
\begin{aligned}
\underline{\bm{\Phi}}_n(t,\xi)=(\underline{u}_n(t,\xi), \overline{v}_n(t,\xi)):={\bf u}(t,\xi;-n,\underline{\bm{\Phi}}_\star(-n,\cdot)),\\
 \overline{\bm{\Phi}}_n(t,\xi)=(\overline{u}_n(t,\xi), \underline{v}_n(t,\xi)):={\bf u}(t,\xi;-n,\overline{\bm{\Phi}}_\star(-n,\cdot)).
\end{aligned}
\end{equation*}
We then have the following result.
\begin{lemma}\label{lower-over-bd-Lemma-2}
For every $n\in\mathbb{Z}^+$, $t\in[-n, 0]$ and $\xi\in\R$, it holds that 
\begin{equation}\label{time-lower-upper-bd-of-n} 
\bm{\underline{\Phi}}_\star(t,\xi) \leq_K \underline{\bm{\Phi}}_n(t,\xi) \leq_K  \underline{\bm{\Phi}}_{n+1}(t,\xi) \leq_K   \overline{\bm{\Phi}}_{n+1}(t,\xi)\le_K  \overline{\bm{\Phi}}_n(t,\xi) \le_K\overline{\bm{\Phi}}_\star(t,\xi).
\end{equation}
In particular, 
\begin{align}\label{lower-upper-bd-of-u-n} 
\bm{\underline{\Phi}}_\star(0,\xi) \leq_K \underline{\bm{\Phi}}_n(0,\xi) \le_K  \overline{\bm{\Phi}}_n(0,\xi) \le_K\overline{\bm{\Phi}}_\star(0,\xi),\quad \forall \xi\in\R,\ \forall\ n\in\mathbb{Z}^+.
\end{align}
\end{lemma}
\begin{proof}
Observe that 
\begin{equation}\label{eq:lowup}
\bm{\underline{\Phi}}_\star(t,\xi) \leq_K \underline{\bm{\Phi}}_n(t,\xi) \le_K  \overline{\bm{\Phi}}_n(t,\xi) \le_K\overline{\bm{\Phi}}_\star(t,\xi),
\end{equation}
 follows from Lemmas \ref{Comparison-principle} and  \ref{enrire-super-sub-sol}, and in turn yields \eqref{lower-upper-bd-of-u-n} by taking $t=0$. Finally,
 \[
 \underline{\bm{\Phi}}_n(t,\xi) \leq_K  \underline{\bm{\Phi}}_{n+1}(t,\xi) \leq_K   \overline{\bm{\Phi}}_{n+1}(t,\xi)\le_K  \overline{\bm{\Phi}}_n(t,\xi),
 \] 
 follows from \eqref{eq:lowup} by taking $t=n-1$ and comparison principle for competitive systems.
\end{proof}

Hence the following functions are well defined
\begin{equation}\label{u^*-under-defin}
\bm{\underline{\Phi}}_*(t,\xi)=(\underline{u}_*(t,\xi),\overline{v}_*(t,\xi)):=\lim_{n\to\infty}\bm{\underline{\Phi}}_n(t,\xi)
\end{equation}
\begin{equation}\label{u^*-over-defin}
\bm{\overline{\Phi}}_*(t,\xi)=(\overline{u}_*(t,\xi),\underline{v}_*(t,\xi))=\lim_{n\to\infty}\bm{\overline{\Phi}}_n(t,\xi)
\end{equation}
Moreover, using estimate for parabolic equations, we have that $\bm{\underline{\Phi}}_n(t,\xi) $ and $\bm{\overline{\Phi}}_n(t,\xi)$ converge respectively to $ \bm{\underline{\Phi}}_*(t,\xi)$ and $\bm{\overline{\Phi}}_*(t,\xi)$ locally uniformly in $ C^{1,2}_{\rm loc}((-\infty,0)\times\R)$. In addition, $\bm{\overline{\Phi}}_*(t,\xi)$  and $\bm{\underline{\Phi}}_*(t,\xi)$ are classical solution of \eqref{rewrite-main-system} on $(-\infty,0]\times\R.$ 

We define 
\[
r(t):=\frac{-1}{\mu}\ln\left(1+\frac{2\varepsilon M}{\mu}e^{\mu t}\right), \quad \forall t\in\R,
\]
and will use the following lemma about $r(t)$ to prove uniqueness of entire solution of \eqref{rewrite-main-system} satisfying  \eqref{asymptotic at -infty}.
\begin{lemma}\label{unqueness lem } The function $r(t)$ holds the following properties.
$$ 
\lim_{t\to-\infty}r(t)=0\quad \text{and}\quad e^{p(t+r(t))}=e^{q(t)},\quad \forall\ t\leq 0.
$$

\end{lemma}

\begin{proof} It is clear that $\lim\limits_{t\to-\infty}r(t)=0$.
Straightforward calculation based on \eqref{p-formula} and \eqref{q-formula} shows that
\[
 e^{p(t+r(t))}=\left( e^{-\mu t} e^{-\mu r(t)} -\frac{\varepsilon M}{\mu}  \right)^{-1}
=\left[e^{-\mu t}\left(1+\frac{2\varepsilon M}{\mu}e^{\mu t}   \right) -\frac{\varepsilon M}{\mu} \right]^{-1}
= \left( e^{-\mu t} +\frac{\varepsilon M}{\mu} \right)^{-1}= e^{q(t)}.
\]
\end{proof}

Now, we give the proof of Theorem \ref{existence-of-entire-solution}.

\medskip

\begin{proof}[Proof of Theorem \ref{existence-of-entire-solution}] 
First, we show the existence. It is clear that $\bm{\underline{\Phi}}_*$ defined in \eqref{u^*-under-defin} (resp. $\bm{\overline{\Phi}}_*$ defined in\eqref{u^*-over-defin}) gives a solution of \eqref{main system} for $(t,\xi) \in (-\infty,0]\times \mathbb{R}$. Moreover, it follows from Lemma \ref{lower-over-bd-Lemma-2} that these functions satisfy the inequality \eqref{asymptotic at -infty}. 
Furthermore, it is standard to extend both of them into entire solutions by solving forward in time with initial data $\bm{\underline{\Phi}}_*(0,\xi)$ (resp. $\bm{\overline{\Phi}}_*(0,\xi)$).
  

Next, we show uniqueness by showing that the pair of super-sub-solutions is deterministic via translation; see \cite[Definition 1]{Chen2005} for details. Let $\bm{\Phi}_{*,i}(t,\xi):=(u_{*,i}(t,\xi),v_{*,i}(t,\xi))$, $i=1,2,$ be entire solutions of \eqref{rewrite-main-system} satisfying \eqref{asymptotic at -infty}. Let $(t,\xi)\in\R\times\R$ be given. For every $n\geq |t|$, and $i=1,2,$ we have 
$$ 
\bm{\Phi}_{*,i}(t,\xi)={\bf u}(t,\xi;-n,\bm{\Phi}_{*,i}(-n,\cdot)).
$$
By Theorem \ref{existence-of-entire-solution} and Lemma \ref{unqueness lem } , it holds that for any $n\in\mathbb{Z}^+$ and $\xi\in\R$,
$$ 
\underline{\bm{\Phi}}_\star(-n,\xi))\leq_K \overline{\bm{\Phi}}_\star(-n,\xi)\leq_K \underline{\bm{\Phi}}_\star(-n+r(-n),\xi).
$$
Thus, for $i,j\in\{1,2\}$, using Lemma \ref{unqueness lem }, we have 
\begin{align*}
\bm{\Phi}_{*,i}(t,\xi)={\bf u}(t,\xi;-n,\bm{\Phi}_{*,i}(-n,\cdot))
\leq_K\, & {\bf u}(t,\xi;-n,\overline{\bm{\Phi}}_\star(-n,\cdot))\cr
=\,\,\,\,&{\bf u}(t,\xi;-n,\underline{\bm{\Phi}}_\star(-n+r(-n),\cdot))\cr
\leq_K\, & {\bf u}(t,\xi;-n, \bm{\Phi}_{*,j}(-n+r(-n),\cdot))\cr
=\,\,\,&\bm{\Phi}_{*,j}(t+r(-n),\xi).
\end{align*}
Letting $n\to\infty$, we conclude from Lemma \ref{unqueness lem } that 
$$ 
\bm{\Phi}_{*,i}(t,\xi)\leq \bm{\Phi}_{*,j}(t,\xi),\quad \forall\ (t,\xi)\in\R^2, \quad \text{and}\ i,j=1,2,
$$
which naturally yields that $\bm{\Phi}_{*,1}(t,\xi)= \bm{\Phi}_{*,2}(t,\xi)$, for every $(t,\xi)\in\R^2$. 
\end{proof}



\subsection{Exponential decay estimates at $x = \pm \infty$}
 In this subsection, we adapt the simplified notation ${\bf u}=(u,v)$ for the entire solution given by Corollary \ref{coro:mai-existence-result}, originally denoted ${\bf u}_\lambda=(u_\lambda, v_\lambda)$, by erasing the sub-index. We aim to determine the exact exponential decay of $u$ at $+\infty$ and $v$ at $x = \pm \infty$.
\begin{proposition}\label{prop:exp}
Let $c \geq 2$ and $\lambda \in (0,\sqrt{r/d})$ such that $g_{d,c,r}(\lambda)>0$ and $g_{d,c,r}(\lambda) \geq r(1-b)$. Let  $0 < \varepsilon \ll 1$ be fixed such that  ${\bf u}=(u,v)$ are given by Corollary \ref{coro:mai-existence-result}. We then have
\begin{equation}\label{eq:expu1}
\lim_{x \to +\infty} e^{\tau_c (x-ct)} u(t,x) = 1 \quad \text{ for each }t \leq 0,
\end{equation}
\begin{equation}\label{eq:expv1}
\lim_{x \to +\infty} e^{\lambda(x - c_v t)} v(t,x) = \varepsilon, \quad \text{ for each }t \in \mathbb{R}.
\end{equation}
where $c_v = d\lambda + \frac{r}{\lambda}$. If, in addition, $b \in (0,1)$, then 
\begin{equation}\label{eq:expv2}
\lim_{x \to +\infty} e^{- \delta_v (x - ct)+ \mu t} v(t,x) = \varepsilon \Upsilon,  \quad \text{ for each }t  \in \mathbb{R}, 
\end{equation}
where we recall that $\delta_v = \frac{1}{2d}\left[ \sqrt{c^2 + 4d(\mu + r(b-1))} - c\right]$, and $\Upsilon$ is given by Lemma \ref{lem:psi}(c) or (d).
\end{proposition}
\begin{proof}
By \eqref{main-asymptotic at -infty}, we have
\begin{equation}\label{eq:u}
\Phi_c(x - ct) + \varepsilon \varphi(x-ct) e^{p(t)} \leq u(t,x) \leq \Phi_c(x - ct) + \varepsilon \varphi(x-ct) e^{q(t)} \quad \text{ for }(t,x) \in (-\infty, 0]\times\mathbb{R}, 
\end{equation}
and
\begin{equation}\label{eq:v}
\varepsilon \psi(x+ (c_v-c)t) e^{q(t)} \leq v(t,x+c_v t) \leq \varepsilon \psi(x+ (c_v-c)t) e^{p(t)} \quad \text{ for }(t,x) \in (-\infty, 0]\times\mathbb{R}, 
\end{equation} where $\psi$ is given by Lemma \ref{lem:psi}, and 
\begin{equation}\label{eq:pqprop2}
\lim_{t \to -\infty} \left(|p(t) - \mu t| + |q(t) - \mu t|\right) = 0. 
\end{equation}
It then follows from \eqref{decay-estimate for phi_c}, \eqref{eq:u} and Remark \ref{rk1} that \eqref{eq:expu1} holds.

We proceed to prove \eqref{eq:expv1}, and note that \eqref{eq:expv2} follows in a similar fashion. 
\begin{claim}\label{claim:vupper}
If there exists $t_0 \in \mathbb{R}$ and $\varepsilon_0 >0$ such that  $v(t_0,x+ c_v t_0) \leq \varepsilon_0 e^{-\lambda x}$ for $x \in \mathbb{R}$, then 
$$
v(t,x + c_v t) \leq \varepsilon_0 e^{-\lambda x} \quad \text{ for }(t,x) \in [t_0,\infty)\times \mathbb{R}.
$$
\end{claim}
To prove this claim, it suffices to observe that $v(t,x+c_v t)$ and $\varepsilon_0 e^{-\lambda x}$ form a pair of  sub and super-solutions of the equation $\widetilde v_t =  d\widetilde v_{xx} + c_v \widetilde v_x+ r\widetilde v$ in the domain $[t_0,\infty)\times\mathbb{R}$.
\begin{claim}\label{claim:vlower}
If there exists $t_0 \in \mathbb{R}$, $\widetilde\lambda \in (\lambda, \min\{\tau_c + \lambda, c/(2d), 2\lambda, \lambda_0\})$ and $D_0,\varepsilon_0 >0$ such that   
$$
v(t_0,x+ c_v t_0) \geq \varepsilon_0 (e^{-\lambda x} - D_0 e^{-\widetilde\lambda x}) \quad \text{ for }x \in \mathbb{R},
$$ 
then there exists $D_1\in (D_0, \infty)$ such that 
$$
v(t,x + c_v t) \geq \varepsilon_0 (e^{-\lambda x }- D_1 e^{-\widetilde\lambda x}) \quad \text{ for }(t,x) \in [t_0,\infty)\times \mathbb{R}.
$$
\end{claim}
First, observe that $v(t,x+ c_v t)$ is a super-solution of 

\begin{equation}\label{eq:vvv}
\widetilde{v}_t = d\widetilde{v}_{xx} + c_v \widetilde{v}_x + r(1-b \min\{1, e^{-\tau_c(x + (c_v-c)t)}\} - \widetilde{v})\widetilde{v}.
\end{equation}
This follows from the second equation of \eqref{main system} and that $u(t,x + c_v t) \leq \Phi_c(x+ (c_v-c)t) \leq \min\{1, e^{-\tau_c(x + (c_v-c)t)}\}$. It remains to show that 
the function  $\max\{0, \varepsilon_0 (e^{-\lambda x }- D_1 e^{-\widetilde\lambda x})\}$ is a sub-solution of \eqref{eq:vvv}, provided $D_1 \gg D_0$. Since this is similar to the proof of Lemma \ref{lem:psi}(a), we omit the details.

By Lemma \ref{lem:psi}(b) and \eqref{eq:v}, there exists $D_0 >0$ such that  for each  $(\bar{t},x) \in \mathbb{R}^- \times \mathbb{R}$, we have
$$
\varepsilon (e^{-\lambda(x+ (c_v-c)\bar t)}- D_0 e^{-\widetilde\lambda (x+ (c_v-c)\bar t)}) e^{p(\bar{t})} \leq v(\bar{t},x + c_v \bar{t}) \leq \varepsilon e^{-\lambda(x+ (c_v-c)\bar t)} e^{q(\bar{t})}. 
$$
By Claims \ref{claim:vupper} and \ref{claim:vlower}, we deduce that for each $\bar{t} < 0$ there exists $D_{\bar t}$ such that 
$$
\varepsilon (e^{-\lambda(x+ (c_v-c)\bar t)}- D_{\bar t} e^{-\widetilde\lambda (x+ (c_v-c)\bar t)}) e^{p(\bar{t})} \leq v(t,x + c_v t) \leq \varepsilon e^{-\lambda(x+ (c_v-c)\bar t)} e^{q(\bar{t})}\quad \text{ for }t  \geq \bar t,\,\, x \in \mathbb{R}.
$$
Using the fact that $\mu = d\lambda^2 - c\lambda + r = \lambda(c_v - c)$, the above can be rewritten as
$$
\varepsilon (e^{-\lambda x -\mu \bar t + p(\bar t)}- D_{\bar t} e^{-\widetilde\lambda (x+ (c_v-c)\bar t) + p(\bar t)})\leq v(t,x + c_v t) \leq \varepsilon e^{-\lambda x - \mu\bar t + q(\bar{t})}\quad \text{ for }t  \geq \bar t,\,\, x \in \mathbb{R}.
$$
%
%
Dividing by $e^{-\lambda x}$ and letting $x \to \infty$, we have
$$
\varepsilon e^{ - \mu \bar t + p(\bar t)} \leq \liminf_{x \to \infty} e^{\lambda x}v(t,x + c_vt) \leq \limsup_{x \to \infty} e^{\lambda x}v(t,x + c_vt)  \leq \varepsilon e^{-\mu\bar t + q(\bar t)} \quad \text{ for }t   \geq \bar t.
$$
Finally, we can take $\bar t \to -\infty$ (recalling \eqref{eq:pqprop2}) to deduce 
$\lim_{x \to +\infty} e^{\lambda x} v(t,x + c_v t) = \varepsilon$ for each $t \in \mathbb{R}$, which is equivalent to \eqref{eq:expv1}.  

Arguing similar for $x \to -\infty$, we can prove \eqref{eq:expv2}. This completes the proof of the proposition.
\end{proof}

\section{Asymptotic behavior of entire solutions.}\label{s:4}

\subsection{Asymptotic behavior of entire solutions of Theorem \ref{Main Tm}.}

In this section, we discuss the asymptotic behavior of the entire solution constructed in the previous section and complete the proof of our main results.
 We first note that the super-solution 
 
$$ 
\overline{\bm{\Phi}}_\star(t,x)=(\overline{u}_\star(t,x),\underline{v}_\star(t,x)):=(\Phi_c(x)+\varepsilon\varphi (x)e^{q(t)},\varepsilon \psi(x)e^{q(t)}), \quad (t,x)\in\R\times\R.
$$
introduced in \eqref{sub-sol-of-main-system} is defined for every $(t,x)\in\R\times\R$. 
 
Throughout this section, we fix $\lambda \in \Lambda_{d,c,r,b}$ and $\varepsilon>0$ so that
 $\bm{\Phi}_*(t,\xi)=(u_*(t,\xi),v_*(t,\xi)) $ and ${\bf u}(t,x)=(u(t,x),v(t,x))$ are, respectively, entire solutions given by Theorem \ref{existence-of-entire-solution} and Corollary \ref{coro:mai-existence-result}, i.e.
 \begin{equation}\label{def-of-main-entire-sol}
{\bf u}(t,x)=\bm{\Phi}_*(t,x-ct), \quad  \quad (t,x)\in\R\times\R,
 \end{equation}
 where we again suppressed the sub-index $\lambda$ for the entire solution given by Corollary \ref{coro:mai-existence-result}.
 
 \subsection{Asymptotic behavior at $t=-\infty$} 
 The following holds.
 \begin{lemma}
 It holds that 
 $$ 
\lim_{t\to-\infty}\|{\bf u}(t,\cdot)-\bm{\Phi}_c(\cdot-ct)\|_{\infty}=0 
 $$
 \end{lemma}
\begin{proof}
The result follows easily from \eqref{main-asymptotic at -infty}. 
\end{proof}

 \subsection{Asymptotic behavior at $t=+\infty$}

  \begin{lemma}\label{New-lm1} Let  $c_v=\frac{d\lambda^2+r}{\lambda}$, it holds that 
 \begin{equation}\label{ZB-e1}
\lim_{t\to\infty}\left[\sup_{x\geq (c+\widetilde{\varepsilon})t}u(t,x) +\sup_{x\geq (c_v+\widetilde{\varepsilon})}v(t,x) \right]=0,\quad \forall\widetilde{\varepsilon}>0.
 \end{equation}
 In particular, \eqref{MT1.1.4} 
  holds.
 \end{lemma}
\begin{proof} Observe that the upper bound in \eqref{main-asymptotic at -infty} holds for all $t \in \mathbb{R}$, so that 
 \begin{equation}\label{Z-0} 
(u(t,x),v(t,x))\leq_K(\Phi_c(x-ct)+\varepsilon\varphi(x-ct)e^{q(t)},\varepsilon\psi(x-ct)e^{q(t)}),\quad \forall\ x\in\R,t\in\R. 
 \end{equation}
 Hence for each $\widetilde\varepsilon >0$, 
 \begin{equation}\label{Z-1} 
\sup_{x\geq (c+\widetilde \varepsilon)t}u(t,x)\leq \Phi_c(\widetilde{\varepsilon}t)\to 0\quad \text{as}\ t\to\infty.
 \end{equation}
Since $\exp(-\lambda(x-c_v)t)$ and $v(t,x)$ form a pair of super-sub-solutions (where $c_v=\frac{d\lambda^2+r}{\lambda}$) of the scalar Fisher-KPP equation
$$ 
v_t=d v_{xx}+rv(1-v),
$$  
 there is a constant $K>1$ such that
 $$v(0,x)\leq \varepsilon\psi(x)e^{p(0)}\leq Ke^{-\lambda x},\forall\ x\in\R,$$
 where the second inequality holds due to the fact that $e^{\lambda x} \psi(x) \to 1$ as $x \to \infty$.
It then follows from the comparison principle for parabolic equations that 
\begin{equation}\label{Z-2}
v(t,x)\leq Ke^{-\lambda(x-c_vt)},\quad \forall t\geq 0, \ x\in\mathbb{R}.
\end{equation}
As a result, the lemma follows from \eqref{Z-1} and \eqref{Z-2}.
\end{proof} 


\begin{lemma}\label{New-lm2}  Let  $c_v=\frac{d\lambda^2+r}{\lambda}$, 
it holds that 
 \begin{equation}\label{ZB-e2}
\limsup_{t\to\infty}\sup_{(c+\widetilde{\varepsilon})t\leq x\leq (c_v-\widetilde{\varepsilon})t}|{\bf u}(t,x)-{\bf e}_2|_1=0,\quad \forall\ 0<\widetilde{\varepsilon}< \frac{c_v - c}{2}.
\end{equation}

\end{lemma}
\begin{proof} It follows from Lemma \ref{New-lm1} that for each $\tilde{\varepsilon}>0$, 
$$\limsup_{t\to\infty}\sup_{ x\geq (c+\tilde{\varepsilon})t}u(t,x) = 0.$$
Furthermore, by \eqref{Z-0}, it holds that 
$$ 
\liminf_{t\to\infty}\sup_{x\ge c_vt}v(t,x)>0.
$$ 
Thus, it is not hard to construct a sub-solution to show that  
$$\liminf_{t\to\infty}\inf_{(c+\tilde{\varepsilon})t\leq x\leq (c_v-\widetilde{\varepsilon})t}v(t,x) > 0,\quad \forall\ 0<\tilde{\varepsilon}\ll 1.$$
Therefore the equation of $v$ can be regarded as an uncoupled equation of KPP-type, and the problem reduces to showing that $1$ is the only entire solution of the KPP equation that is bounded below by a positive constant. 

Since the proof of Lemma \ref{New-lm2} follows from an almost same argument as the one in \cite[Proposition 3.1]{LeoLam2018},which in turn follows from the arguments by Ducrot, Giletti and Matano in \cite{Ducrot2019}, we omit the proof here and refer interested readers to \cite{LeoLam2018,Ducrot2019} for details.
\end{proof}
\begin{lemma}\label{lem:uchiyama}
Suppose  $g_{d,c,r}(\lambda) > r\max\{1-b,0\}$ and $\varepsilon$ be fixed such that $(u,v)$ is the entire solution specified by Corollary \ref{coro:mai-existence-result}. There exists $h_0\in \mathbb{R}$ such that 
for any $\widetilde\varepsilon>0$, 
\begin{equation}\label{eq:uchiyama}
\lim_{t \to \infty} \sup_{x > (c + \widetilde\varepsilon)t} |v(t,x) - \Phi_{c_v}(x-c_v t - h_0)| = 0.
\end{equation}
In fact, we deduce from \eqref{eq:expv1} that $h_0 = -\frac{1}{\lambda} \log \varepsilon$.
\end{lemma}
\begin{proof}
First, observe that $\sup_{x > (c + \varepsilon)t} u(t,x) \to 0$  exponentially as $t \to +\infty$. Based on the exact exponential decay of $v(0,x)$ at $x = +\infty$; see Proposition \ref{prop:exp}, we  apply \cite[Theorem 8.2 or 9.3]{Uch} to yield \eqref{eq:uchiyama}.
\end{proof}

\subsubsection{Monostable case}
Now, we present the proof of Theorem \ref{Main Tm} by establishing the large time behavior of the entire solutions in the monostable cases:  
 $$
 {\bf (1)}  \quad 0 < a,b < 1, \quad \text{ and }\quad {\bf(3)} \quad 0 < a < 1 < b.
 $$ 
 
 \begin{proof}[Proof of Theorem \ref{Main Tm} for cases {\bf(1)} and {\bf(3)}]
Recall the exponential decay estimates of $(u,v)(0,x)$ at $x = \pm \infty$ as described in Remark \ref{rmk:jj} and Proposition \ref{prop:exp}. For case {\bf(1)} we apply \cite[Theorem 1.3]{LLL2019} to prove \eqref{MT1.1.1} - \eqref{MT1.1.4}, whereas for case {\bf(2)} we utilize either \cite[Theorem 6.1]{LLL2019} or \cite[Theorem 1.3]{LeoLam2018} to yield \eqref{e:MT1.3.1} - \eqref{MT1.3.3}. In the latter case, it suffices to observe that for $t \geq 0$, our solution $(u,v)$ can be controlled by the pair of super-sub-solutions constructed in \cite[Propositions 1.4 and 1.6]{LeoLam2018}. Finally, \eqref{MT1.1.5} and \eqref{MT1.3.4} follows from Lemma \ref{lem:uchiyama}.
 \end{proof}

\subsubsection{Bistable case.}  In this subsection we complete the proof of Theorem \ref{Main Tm} by establishing the large time behavior of the entire solutions in the bistable case $${\bf (2)}  \quad  a,b > 1.$$

We note that Lemma \ref{New-lm2} provides an upper bound for the spreading speed of the species $u(t,x)$, and Lemmas \ref{New-lm1} and \ref{New-lm2} show that the faster but weaker competitor $v(t,x)$ spread at the speed $c_vt$.

  As mentioned above, to complete the proof of Theorem \ref{Main Tm} in case {\bf2}, we follow the techniques developed in \cite{Car2018} and \cite{Peng2019}. More specifically, we first introduce some useful functions 
\begin{equation}\label{eq:xipq}
 \xi(t)  =  \xi_0 e^{-\delta_1 t}, \quad P(t) = P_0 e^{-\delta_1 t}, \quad \text{ and }\quad Q(t) = Q_0 e^{-\delta_1 t},
 \end{equation}
 where $\delta_1, P_0, Q_0>0$ and $\xi_0<0$ are constants.  
 \begin{lemma}\label{New-lm4}
 For each  $\delta_1>0$ sufficiently small, there exist $P_0, Q_0>0$ and $\xi_0 <0$  such that $(\overline{u}(t,x),\underline{v}(t,x))$ on $\R^+\times\R$ given by
$$
\begin{cases}
\overline{u}(t,x):=\max\{0,\varphi_{uv}(x-C_{uv}t-\xi(t))-Q(t)\} \cr
\underline{v}(t,x):=\min\{1,\psi_{uv}(x-C_{uv}t-\xi(t))+P(t)\}
\end{cases}
 $$
where $\bm{\Phi}_{uv}:=(\varphi_{uv},\psi_{uv})$ is the traveling wave solution to \eqref{eq:TW-eq1} with speed $C_{uv}$,
satisfies
$$
\begin{cases}
\mathcal{A}_{1,C_{uv}}(\overline{u},\underline{v})(t,x)\geq 0 \cr
\mathcal{A}_{1,C_{uv}}(\overline{u},\underline{v})(t,x)\le 0
\end{cases} \quad  \text{ in the weak sense for }(t,x) \in \mathbb{R}^+\times \mathbb{R}.
 $$
 \end{lemma}
\begin{proof}
Define   
$$
\underline{S}(\xi) = \frac{1}{r}\varphi_{uv}(\sqrt{d} \xi) \quad \text{ and }\quad \overline{R}(\xi) =\psi_{uv}(\sqrt{d}\xi),
$$
then $(\underline{S}, \overline{R})$ satisfies
$$
\begin{cases}
\underline{\delta}_0 \underline{S}'' - \underline{c}_{SR} \underline{S}' + \underline{S}(\underline\alpha - \underline{S} - \underline{r}_0\overline{R}) = 0 \cr
\overline{R}'' - \underline{c}_{SR} \overline{R}' + \overline{R}(1- \overline{R} - \underline\beta_0 \underline{S}) = 0
\end{cases}
$$
where 
$$
\underline\delta_0 = \frac{1}{rd}, \quad  \underline{c}_{SR} = \frac{C_{uv}}{r \sqrt{d}}, \quad \underline\alpha  = \frac{1}{r}, \quad \underline{r}_0 = \frac{a}{r}, \quad \text{ and }\quad \underline\beta_0 = rb.
$$
And we may argue exactly the same as in \cite[Lemma 7]{Car2018}. 
\end{proof}

\begin{proof}[Proof of Theorem \ref{Main Tm} for case \bf{(2)}]
Firstly, the proof of \eqref{e:MT1.2.3} are exactly the same as in case (1). Now, by Lemma \ref{New-lm2} we have
\begin{equation}\label{eq:q0}
\lim_{t \to \infty}\sup_{ (c + \varepsilon)t < x < (c_v - \widetilde\varepsilon)t} |{\bf u} - {\bf e}_2|_1 = 0, \quad \text{ for each }\widetilde\varepsilon >0,
\end{equation}
which shows that part of \eqref{e:MT1.2.2} holds.
It remains to prove \eqref{e:MT1.2.1} and the rest of \eqref{e:MT1.2.2}.

Consider the solution $\hat{{\bf u}}=(\hat{u}, \hat{v})$ of \eqref{main system} in the domain $(t,x) \in \mathbb{R}^+\times\mathbb{R}$ with initial data $(\hat{u}_0,\hat{v}_0)$ such that $\hat{u}_0$ is compactly supported with $0 \leq \hat{u}_0(x) \leq u(x,0)$, and  $\hat{v}_0 \equiv 1$. By \cite[Theorem 1]{Peng2019}, there exists $h_2\in \mathbb{R}$ such that
\begin{equation}\label{eq:q1}
\lim_{t \to \infty}\sup_{x \geq 0} |\hat{{\bf u}}(t,x) - \bm{\Phi}_{uv}(x-C_{uv}t - h_2)|_1 =0.
\end{equation}
Note that we have
\begin{equation}\label{eq:q2}
\hat{{\bf u}}(t,x)\leq_{K} {\bf u}(t,x)\quad \text{ for }(t,x) \in \mathbb{R}^+ \times\mathbb{R}.
\end{equation}
In particular, for each $c^-, c^+$ such that $c^- < c^+ < C_{uv}$, we have
\begin{equation}\label{eq:q2b}
\lim_{t\to\infty} \inf_{c^- t < x < c^+ t}u(t,x)   \geq 1 \quad \text{ and }\quad \lim_{t\to\infty} \sup_{c^- t < x < c^+ t}v(t,x) =0.
\end{equation}

Furthermore, exploiting \eqref{eq:q0}, we can repeat the proof of \cite[Lemmas 4.6 and 4.7]{Peng2019} to show that, for each $\hat{c} \in (c,c_v)$, there exists $C_1, \delta_1, T_1$ such that
\begin{equation}\label{eq:q3}
u(t,\hat{c}t) \leq C_1 e^{-2\delta_1 t}, \quad  v(t,\hat{c}t) \geq 1-C_1 e^{-2\delta_1 t}, \quad \text{ for }t \geq T_1.
\end{equation}
By using \eqref{eq:q2b} and possibly enlarging $\delta_1$ and $T_1$, it is not difficult to show that for each $c^\sharp \in (-\infty, C_{uv})$, 
\begin{equation}\label{eq:q4}
\quad u(t,c^\sharp t) \leq 1, \quad v(t,c^\sharp t) \geq C_1 e^{-2\delta_1 t} \quad \text{ for }t \geq T_1,
\end{equation}
the latter follows from a sub-solution $\underline{v}_1$ for the equation of $v$ of the form 
$$\underline{v}_1(x,t) =\begin{cases} \varepsilon e^{-2\delta_1 t} \cos(\lambda(x-c^\sharp t)), &|x - c^\sharp t| < \pi/(2\lambda),\\
0, & |x-c^\sharp t| \geq \pi/(2\lambda), \end{cases}$$ 
where $\delta_1 \in ( r(b-1),\infty)$, and  $\lambda = \frac{1}{2d} \sqrt{|c^\sharp|^2 - 4d(\delta + r - rb)}$. Taking advantage of  the estimates \eqref{eq:q3} and \eqref{eq:q4}, one can then apply  the comparison principle to prove that
\begin{equation}\label{eq:q5}
{\bf u}\leq_{K}(\overline{u}, \underline{v}),\quad \text{ for }c^\sharp t \leq x \leq \hat{c} t, \,\, t \geq T_1,
\end{equation}
where $(\overline{u}, \underline{v})$ are given in Lemma \ref{New-lm4}. Passing to a sequence $t_n \to \infty$, we may assume ${\bf u}_n(t,x):= {\bf u}(t + t_n, x-C_{uv}t_n)$ converges in $C^{1,2}_{loc}(\mathbb{R}^2)$ to some ${\bf u}_\infty(t,x):=(u_\infty(t,x), v_\infty(t,x))$. 
By \eqref{eq:q1}, \eqref{eq:q2} and \eqref{eq:q5}, there exists $h_3>0$ such that
\begin{equation}\label{eq:q6}
\varphi_{uv}(x + h_3) \leq u_\infty(t,x) \leq \varphi_{uv}(x - h_3) \quad \text{ and }\quad\psi_{uv}(x - h_3) \leq v_\infty(t,x) \leq \psi_{uv}(x + h_3)  \quad \text{ for } (t,x) \in \mathbb{R}^2.
\end{equation}
We may then argue similarly as in the proof of \cite[Section 3.2]{Peng2019} to obtain \eqref{e:MT1.2.1}. We omit the details.
\end{proof}

\section{Proof of Theorem \ref{MT-3}}\label{s:5}

In this  section we outline the proof of Theorem  \ref{MT-3}.  Suppose that $d>0$, $r>0$, $0<a<1<b$ and $c_v\geq 2\max\{\sqrt{rd},\sqrt{a}\}$. Denote $$ 
\lambda_v:=\frac{1}{2}\left(c_v-\sqrt{c_v^2-4rd} \right),\quad \lambda:=\frac{1}{2}\left(c_v-\sqrt{c_v^2-4a} \right)\quad \text{and}\quad \varphi_{\lambda}(x):=e^{-\lambda x}, \quad x\in\R.
$$

Then define
\begin{equation}\label{new-eqz2'}
\mu := g_{1,c,1}(\lambda) = r(1-a) >0.
\end{equation}
By similar arguments to the proof of Lemma \ref{existence-of-unstable-eigenfunction} where $g_{1,c,1}(\lambda) = (1-a)>0$, we can prove the following result.

\begin{lemma} \label{existence-of-unstable-eigenfunction-2} Suppose that $d>0$, $r>0$, $0<a<1<b$ and $c_v> 2\max\{\sqrt{rd},\sqrt{a}\}$. Then there uniquely exists  $(\widehat{\varphi},\widehat{\psi})\in C^{2}(\R)$  such that for all $x\in\R$
\begin{equation}\label{eigenv-eq-3}
\begin{cases}
(1-a)\widehat{\varphi}=\widehat{\varphi}_{xx}+c_v\widehat{\varphi}_{x}+(1-a\Psi_{c_v}(x))\widehat{\varphi},\cr
(1-a)\widehat{\psi}=d\widehat{\psi}_{xx}+c_v\widehat{\psi}_{x}+r(1-2\Psi_{c_v}(x))\widehat{\psi}-rb\Psi_{c_v}\widetilde{\varphi},\cr
\widehat{\psi}(\pm\infty)=0, \quad \text{ and }\quad  \widehat{\psi}<0,\cr
{\displaystyle \lim_{x\to\infty}}\frac{\widehat{\varphi}(x)}{e^{-\lambda x}}=1,\quad {\displaystyle \sup_{x \in \mathbb{R}} \hat\varphi } < +\infty,  \quad \text{ and }\quad \widehat{\varphi}>0.
\end{cases}
\end{equation}
where $\lambda=\frac{1}{2}\left(c_v-\sqrt{c_v^2-4a} \right)$. Moreover, there exists $\Upsilon>0$ such that
$\displaystyle  \lim_{t\to-\infty}\widehat{\varphi}(x)= \Upsilon.$
\end{lemma}

Using this result, we can again proceed as in Section \ref{s:3} and establish the following result.

\begin{theorem}\label{exist-tm-3'}
Suppose that $d>0$, $r>0$, $0<a<1<b$ and $c_v\geq 2\max\{\sqrt{rd},\sqrt{a}\}$. Let $\widehat{\bm{\Phi}}:=(\widehat{\varphi},\widehat{\psi})$ be the solution of \eqref{eigenv-eq-3} given by Lemma \ref{existence-of-unstable-eigenfunction-2}. Then the following statements hold. 
\begin{itemize}
\item[(a)] 
For each $0 < \varepsilon \ll 1$, there is a unique entire solution ${\bf u}(t,x):=(u(t,x),v(t,x))$ of \eqref{main system} satisfying 
\begin{equation}\label{main-asymptotic at -infty-3}
\bm{\Psi}_{c_v}(x-c_vt)+\varepsilon\widehat{\bm{\Phi}}(x-c_vt)e^{\widehat{p}(t)}\leq_K  {\bf u}(t,x)
\leq_K  \bm{\Psi}_{c_v}(x-c_vt)+\varepsilon\widehat{\bm{\Phi}}(x-c_vt)e^{\widehat{q}(t)},
\end{equation}
for $(t,x) \in (-\infty, 0] \times \R$, where 
 $$ 
 \lim_{t\to-\infty}|\widehat{p}(t)-(1-a) t|=\lim_{t\to-\infty}|\widehat{q}(t)-(1-a) t|=0.
 $$ 
\item[(b)] Furthermore,
\begin{equation}\label{eq:expp}
\lim_{x \to +\infty} e^{\lambda (x - c_{u,3}t)} u(t,x) = \varepsilon, \quad \text{ for each }t \in \mathbb{R}. 
\end{equation}
where $c_{u,3}>c_v$ and $\lambda \in (0,1)$ are given by \eqref{new-eqz2}.
\end{itemize}
\end{theorem}
\begin{remark}
The function $\widehat{p}(t)$ is defined for all time $t\in\R$, strictly increasing and bounded, and first inequality of \eqref{main-asymptotic at -infty-3} holds in fact for all $t \in \mathbb{R}.$ 
\end{remark}

\medskip

To complete the proof of Theorem \ref{MT-3}, it remains to show that the the entire solution ${\bf u}(t,x)$ provided by Theorem \ref{exist-tm-3'} satisfies the desired asymptotic behaviors at $t\approx\pm\infty$.

It is clear from \eqref{main-asymptotic at -infty-3} that \eqref{MT3.1.1} holds.  Note also from \eqref{main-asymptotic at -infty-3} that $v(0,x) \leq \Phi_{c_v}(x)$ for all $x$, so that by comparison, we have
$$
v(t,x) \leq \Phi_{c_v}(x-c_v t) \quad \text{ for } (t,x) \in \mathbb{R}^+\times \mathbb{R}.
$$ 
Hence, 
\begin{equation}\label{eq:ww-1}
\limsup_{t\to\infty}\sup_{x\ge (c_v+\widetilde{\varepsilon})t}v(t,x)\leq \limsup_{t\to\infty}\sup_{x\ge (c_v+\widetilde{\varepsilon})t}\Phi_{c_v}(\widetilde{\varepsilon}t)=0,\quad \forall\ 0<\widetilde{\varepsilon}\ll 1.
\end{equation}
Since
\begin{equation}\label{eq:ww-2}
u_t\geq u_{xx}+u(1-a-u)
\end{equation}
and $\liminf_{x\to-\infty}u(0,x)>0$, which is due to \eqref{main-asymptotic at -infty-3} and that $\displaystyle \lim_{x \to -\infty} \hat\varphi(x) >0$, we then conclude from spreading speed properties for Fisher-KPP equations that 
\begin{equation}\label{eq:ww-3}
\liminf_{t\to\infty}\inf_{x\leq (2\sqrt{1-a}-\widetilde{\varepsilon})t}u(t,x)\geq 1-a,\quad \forall\ 0<\widetilde{\varepsilon}\ll 1.
\end{equation}
Next, observe from Theorem \ref{exist-tm-3'}(b) that
\begin{equation}\label{eq:5.14}
\lim_{x\to\infty}e^{\lambda x}{u(0,x)} = \varepsilon>0, 
\end{equation}
and $c=\lambda+\frac{1}{\lambda}=c_v+\frac{1-a}{\lambda}>c_v$, where $0<\lambda<1$. Now, since $v \to 0$ in the moving coordinate with speed greater than $c_v$ by \eqref{eq:ww-1}, and that $u$ spreads in the absense of $v$ at speed $c = \lambda + \frac{1}{\lambda}$, we argue as in the proof of Lemma \ref{New-lm2} to show that 
\begin{equation}\label{eq:5.15} 
\lim_{t\to\infty}\sup_{(c_{v}+\widetilde{\varepsilon})t\leq x\leq (c-\widetilde{\varepsilon})t }|u(t,x)-1|=0,\quad \forall\ 0<\widetilde{\varepsilon}\ll 1,
\end{equation}
which, combined with 
\eqref{eq:ww-3} and  comparison principle for scalar parabolic equations, yields that 
\begin{equation}\label{eq:ww-4}
\liminf_{t\to\infty}\inf_{x\leq (c-\widetilde{\varepsilon})t}u(t,x)\geq 1-a.
\end{equation}
By \eqref{eq:ww-4}, and using $b>1>a$, we can use the classification of entire solution of \eqref{main system}; see \cite[Lemma 2.3]{LLL2019a}, to show that 
\begin{equation}\label{eq:ww-5}
\lim_{t\to\infty}\inf_{x\leq (c-\widetilde{\varepsilon})t}|{\bf u}(t,x)-{\bf e}_1|_1=0,\quad \forall\ 0<\widetilde{\varepsilon}\ll 1.
\end{equation}
Hence $\lim_{t\to\infty} \sup_{x \in \mathbb{R}} |v(t,x)| = 0$ in an exponential manner follows from \eqref{eq:ww-1} and \eqref{eq:ww-5}, so that the equation $u$ reduces to the KPP equation (with exponentially small in $t$ error terms) as $t \to +\infty$. Finally, note that $u(0,x)$ satisfies  \eqref{eq:5.14} and \eqref{eq:ww-4}, so we can apply \cite[Theorem 8.2 or 9.3]{Uch} to yield \eqref{MT3.1.2}. 
%
%
%
This completes the proof of Theorem \ref{MT-3}.

\section{Proof of Theorem \ref{MT2}}\label{s:6}

In this  section we outline the proof of Theorem  \ref{MT2}. Let $d>0$, $r>0$, $0<a,b<1$ and $c\ge 2$ be given such that \eqref{new-eqz1} holds. 

Therefore, appying Lemma \ref{existence-of-unstable-eigenfunction} for the case $\mu = g_{d,c,r}(\lambda) = r(1-b) >0$, we have the following result.

\begin{lemma} \label{existence-of-unstable-eigenfunction-3} Suppose that $d>0$, $r>0$, $0<a,b<1$, and $c> 2\max\{1,\sqrt{drb}\}$ are given. Set 
$$ 
\widetilde\lambda \in (\lambda_v, \lambda_v + \tau_c) \quad \text{ and }\quad \lambda_v=\frac{1}{2{d}}\left(c-\sqrt{c^2-4drb} \right).
$$
%
Then there uniquely exists  $\bm{\Phi}:=(\varphi,\psi)\in C^{2,b}(\R)$ satisfying 
\begin{equation}\label{eigenv-eq-4}
\begin{cases}
r(1-b)\varphi=\varphi_{xx}+c\varphi_{x}+(1-2\Phi_{c}(x))\varphi - a \Phi_c \psi,\quad &\text{ in }\R,\cr
r(1-b)\psi=d\psi_{xx}+c\psi_{x}+r(1-b\Phi_{c}(x))\psi,\quad &\text{ in }\R,\cr
 \varphi<0<\psi, &\text{ in }\R\cr
\varphi(\pm\infty)=0,\quad \text{ and }\quad e^{-\lambda_v x} - D e^{-\widetilde\lambda x} \leq \psi(x) \leq e^{-\lambda_v x} \quad &\text{ for }x \gg 1. 
\end{cases}
\end{equation}
Moreover, there exists $\Upsilon>0$ such that $\displaystyle \lim_{x \to -\infty} \psi(x) = \Upsilon$. 
\end{lemma}

Using this result, we can again proceed as in Section \ref{s:3} and establish the following result.

\begin{theorem}\label{exist-tm-4'}
Suppose that $d>0$, $0<a,b< 1$, and $c\ge 2\max\{1,\sqrt{drb}\}$ be given. Let $\bm{\Phi}$ be the solution of \eqref{eigenv-eq-3} given by Lemma \ref{existence-of-unstable-eigenfunction-3}. Then for each $0<\varepsilon \ll 1$, there is a unique entire solution ${\bf u}(t,x):=(u(t,x),v(t,x))$ of \eqref{main system} satisfying 
\begin{equation}\label{main-asymptotic at -infty-4}
\bm{\Phi}_{c}(x-ct)+\varepsilon\bm{\Phi}(x-ct)e^{p(t)}\leq_K  {\bf u}(t,x)
\leq_K  \bm{\Phi}_{c}(x-ct)+\varepsilon\bm{\Phi}(x-ct)e^{q(t)},
\end{equation}
for every $t\leq 0,\ x\in\R$, where $0<\varepsilon\ll 1$ and 
 $$ 
 \lim_{t\to-\infty}|p(t)-r(1-b) t|=\lim_{t\to-\infty}|q(t)-r(1-b) t|=0.
 $$ 
Moreover, $q(t)$ is defined for all time $t\in\R$, strictly increasing and bounded, and second inequality of \eqref{main-asymptotic at -infty-4} holds for $t \geq 0$ as well.
\end{theorem}

\begin{proof}[Proof of Theorem \ref{MT2}]
To complete the proof of Theorem \ref{MT2}, it remains to show that the entire solution $(u(t,x),v(t,x))$ provided by Theorem \ref{exist-tm-4'} satisfies the desired asymptotic behaviors at $t\approx\pm\infty$.

It is clear from \eqref{main-asymptotic at -infty-4} that \eqref{MT2.1} holds. Note also from \eqref{main-asymptotic at -infty-4} that 
\begin{equation}
\inf_{x\leq x_0 }\min\{u(0,x),v(0,x)\}>0,\quad \forall x_0\in\R.
\end{equation}
Furthermore, By Proposition \ref{prop:exp}, we have
\begin{equation}
\lim_{x\to\infty}e^{\tau_c x} {u(0,x)} = 1,
\end{equation}
where $\tau_c=\frac{1}{2}\left(c-\sqrt{c^2-4}\right)$ and 
\begin{equation}
\lim_{x\to\infty}e^{\lambda_v x}{v(0,x)} = \varepsilon>0, 
\end{equation}
where $\lambda_v=\frac{1}{2}\left(c-\sqrt{c^2-4rb}\right)$. We note that $c_v=\lambda_v+\frac{r}{\lambda}>\lambda_v+\frac{rb}{\lambda_v}=c=\tau_c+\frac{1}{\tau_c}$.  

Hence, we can deduce the (rightwards) spreading speed $c_{u,1}$ of ${\bf e}_*$ by the results in \cite{LLL2019}, this establishes \eqref{MT2.2}, and that
$$
\lim_{t\to\infty} \sup_{(c_{u,1}+\widetilde\varepsilon)t < x < (c_{v,3} - \widetilde\varepsilon)t} |{\bf u}(t,x) - {\bf e}_2|_1 = 0.
$$
Then, we can apply \cite[Theorem 8.2 or 9.3]{Uch} to yield \eqref{MT2.4}. We omit the details. 
\end{proof}

\begin{appendices}
\section{Alternative Proof of Lemma \ref{lem:psi}}
In this section, we give an alternative proof of Lemma \ref{lem:psi}. In fact, the result we prove here is more general. We first introduce the operator
\[
\begin{matrix}
L: & H^2(\R) & \longrightarrow & L^2(\R)\\
& v & \longmapsto & d v_{\xi\xi}+cv_\xi+r(1-b\Phi_c)v
\end{matrix}
\]
and then the results in Lemma \ref{lem:psi} are now spectral properties of the linear operator $L$. More specifically, we have the following lemma.
\begin{lemma} Given $c\geq 2$ and $\mu\in\R$, the eigenvalue-eigenfunction problem 
\begin{equation}\label{e:Lop}
(L-\mu)\psi=0. 
\end{equation}
admits the following properties.
\begin{itemize}
\item If $\mu\in(r(1-b), r)$, up to scalar multiplication, there exists a unique solution to \eqref{e:Lop} in $H^2(\R^2)$. Furthermore, if $\mu\in (\max\{r(1-b), r-\frac{c^2}{4d}\}, r)$, then the solution to \eqref{e:Lop}, $\phi(\xi)\in H^2(\R)\cap C^2(\R)$, is nonzero everywhere (thus can be chosen to be positive) and admits the following asymptotic property,
\[
 \lim_{\xi\to+\infty}\frac{\phi_\xi}{\phi}=-\lambda:=-\frac{c-\sqrt{c^2-4d(r-\mu)}}{2d}, 
\]
where $\lambda\in(0,\frac{c}{2d})$ solves $\mu=d\lambda^2-c\lambda+r$.
\item If $\mu\in(-\infty, r(1-b))\cup(r,+\infty)$, there is no solution to \eqref{e:Lop} in $H^2(\R)$. 
\item If $\mu=r(1-b)>r-\frac{c^2}{4d}$, then up to scalar multiplication, there exists a unique solution $\phi$ to \eqref{e:Lop} in $C^2(\R)$ such that 
\[
\lim_{\xi\to+\infty}\frac{\phi_\xi}{\phi}=-\lambda, \qquad \lim_{\xi\to-\infty}\phi=\Upsilon, \quad \textrm{for some }\Upsilon\in\R.
\]
\end{itemize}
\end{lemma}
\begin{proof}
We study the more general case $\mu\in\mathbb{C}$ and introduce the vector $\bm{\Phi}:=(\phi, \phi_\xi)$. The equation \eqref{e:Lop} can be written as 
\begin{equation}\label{e:LopV}
\bm{\Phi}_\xi=A(\xi,\mu)\bm{\Phi}:=\begin{pmatrix}0 & 1 \\ \frac{1}{d}[\mu-r(1-b\Phi_c)] & \frac{c}{d} \end{pmatrix}\bm{\Phi}.
\end{equation}
where the matrix $A(\xi,\mu)$ approaches constant matrices as $\xi\to\pm\infty$; that is,
\[
A_+=\lim_{\xi\to+\infty}A(\xi,\mu)=\begin{pmatrix}0 & 1 \\ \frac{1}{d}(\mu-r) & \frac{c}{d} \end{pmatrix}, \quad 
A_-=\lim_{\xi\to-\infty}A(\xi,\mu)=\begin{pmatrix}0 & 1 \\ \frac{1}{d}[\mu-r(1-b)]& \frac{c}{d} \end{pmatrix}.
\]
It is well known that the operator $L-\mu$, where $\mu\in\mathbb{C}$, is Fredholm if and only if both $A_+$ and $A_-$ are hyperbolic; see \cite[Theorem 3.1.11]{KapPro} for more details. Introducing the characteristic polynomial 
\begin{equation}
P(\lambda;\xi)=d\lambda^2+c\lambda+r(1-b\Phi_c(\xi)),
\end{equation}
and noting that $A_\pm$ are hyperbolic if and only if 
\[
\mu\not\in \Gamma_+\bigcup \Gamma_-, \quad\text{where } \Gamma_+:=\{P_+(\mathrm{i}k)\mid k\in\R\} \text{ and } \Gamma_-:=\{P_-(\mathrm{i}k) \mid k\in\R\},
\]
where $P_+(\lambda)=\lim\limits_{\xi\to+\infty}P(\lambda;\xi)=d\lambda^2+c\lambda+r$ and $P_-(\lambda)=\lim\limits_{\xi\to-\infty}P(\lambda;\xi)=d\lambda^2+c\lambda+r(1-b)$,
we conclude that $L-\mu$ is Fredholm if and only if $\mu\not\in \Gamma_+\bigcup \Gamma_-$. The Fredholm boundaries, $\Gamma_+\bigcup \Gamma_-$, divide the complex plane into 3 simply connected regions, 
\[
\Omega_1:=\{\alpha+\mathrm{i}\beta\mid \alpha>-d\left( \frac{\beta}{c}\right)^2+r, \beta\in\R\}, \, \Omega_2:=\{\alpha+\mathrm{i}\beta\mid \alpha<-d\left( \frac{\beta}{c}\right)^2+r(1-b), \beta\in\R\}, \, \Omega_3:=\mathbb{C}\backslash \overline{\Omega_1\cup\Omega_2};
\]
see Figure \ref{f:Fred} for an illustration. Furthermore, the Morse indice of $A_\pm$, that is, the dimension of unstable space associated to $A_\pm$,respectively denoted as $i_\pm(\mu)$, takes distinctive values in $\Omega_i$'s, $i=1,2,3$; that is,
\[
i_+(\mu)=\begin{cases}1, & \mu\in\Omega_1,\\ 0, & \mu\not\in\overline{\Omega_1} \end{cases}, \qquad  i_-(\mu)=\begin{cases}1, & \mu\not\in\overline{\Omega_2},\\ 0, & \mu\in\overline{\Omega_2} \end{cases}, 
\]
which yields that the Fredholm index of $L-\mu$, denoted as $\textrm{ind}(L-\mu)$, takes the following values.
\[
\mathrm{ind}(L-\mu)=i_-(\mu)-i_+(\mu)=\begin{cases}0, & \mu\in\Omega_1, \\ 1,  & \mu\in\Omega_3, \\0, & \mu\in\Omega_2. \end{cases}
\]
\begin{figure}
\centering
\includegraphics[width=0.6\textwidth]{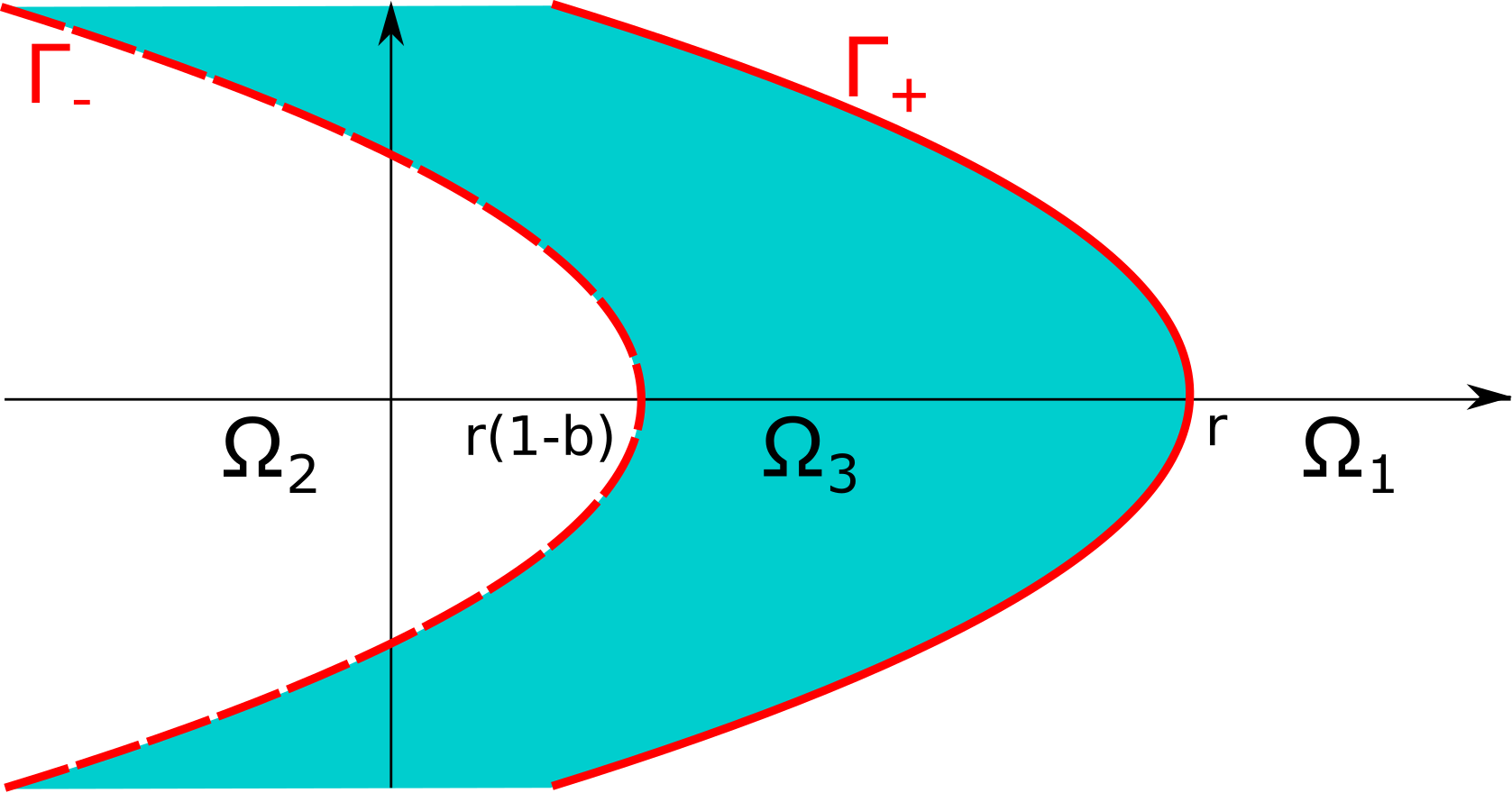}
\caption{A depiction of the sets $\{\Omega_i\}_{i=1}^{3}$ separated by the Fredholm boundaries $\Gamma_+$ (red, solid) and $\Gamma_-$ (red, dashed).}
\label{f:Fred}
\end{figure}
According to the exponential dichotomy theory in \cite{Cop}, in the case when $\mu\in\Omega_2$, we conclude from $i_\pm(\mu)=0$ that there is no solution $\bm{\Phi}$ to \eqref{e:LopV} in $H^1(\R^2)$. As a result, there is no solution $\phi$ to \eqref{e:Lop} in $H^2(\R)$. Similarly, in the case when $\mu\in\Omega_3$, we conclude from $i_+(\mu)=0$ and $i_-(\mu)=1$ that, up to scalar multiplication, there exist a unique solution $\phi$ to \eqref{e:LopV} in $H^2(\R)$. Moreover, for any $\mu\in\Omega_3$, the polynomial $P_-(\lambda)=\mu$, that is,
\[
d\lambda^2+c\lambda+r(1-b)-\mu=0,
\]
admits two distinctive roots
\[
\widetilde{\lambda}_\pm:=\frac{-c\pm\sqrt{c^2-4d(r-rb-\mu)}}{2d},
\]
which, when $\mu\in (\max\{r(1-b), r-\frac{c^2}{4d}\}, r)$, admits $\widetilde{\lambda}_-<0<\widetilde{\lambda}_+$. Moreover, the solution $\phi$ to \eqref{e:Lop} as $\xi\to-\infty$ satisfies
\begin{equation}\label{e:tan-}
\lim_{\xi\to-\infty} \frac{\phi_\xi}{\phi}=\widetilde{\lambda}_+.
\end{equation}
Similarly,  the polynomial $P_+(\lambda)=\mu$, that is, $d\lambda^2+c\lambda+r-\mu=0$, admits two distinctive roots 
\[
\lambda_\pm:=\frac{-c\pm\sqrt{c^2-4d(r-\mu)}}{2d},
\]
which, when $\mu\in (\max\{r(1-b), r-\frac{c^2}{4d}\}, r)$, admits $\lambda_-<\lambda_+<0$.
To show that $\phi $ is nonzero everywhere and $\lim\limits_{\xi\to+\infty}\frac{\phi_\xi}{\phi}=\lambda_+$, we introduce the polar coordinates $(r,\theta)\in [0,+\infty)\times\mathbb{T}_{2\pi}$ satisfying 
\[
\bm{\Phi}=\begin{pmatrix}\phi \\ \phi_\xi \end{pmatrix}=\begin{pmatrix}r\cos\theta \\ r\sin\theta \end{pmatrix},
\]
and rewrite \eqref{e:LopV} in the polar coordinates as
\begin{equation}\label{e:LopR}
\begin{cases}
r_\xi=r\left\{ \left[ 1+\frac{\mu}{d} -\frac{r}{d}(1-b\Phi_c)\right]\cos\theta-\frac{c}{d}\sin\theta\right\}\sin\theta,\\
\theta_\xi=F(\theta,\xi):=-\frac{1}{d}\left(P(\tan\theta;\xi)-\mu\right)\cos^2\theta.
\end{cases}
\end{equation}
Noting that for any $\xi\in\R$,
\[
F(\pi/2, \xi)=-\sin^2(\pi/2)=-1<0, \qquad F(\arctan^{-1}(\lambda_+),\xi)=\frac{rb\Phi_c(\xi)}{d(\lambda_+^2+1)}>0, 
\]
it is straightforward to see that the interval $(\arctan^{-1}(\lambda_+), \pi/2)$ is forward-invariant for $\theta$.  Recalling the limiting behavior of $\phi$ as $\xi\to-\infty$, \eqref{e:tan-}, we conclude
\[
\frac{\phi_\xi(\xi)}{\phi(\xi)}\in (\lambda_+,\infty), \quad \forall \xi\in\R,
\]
which shows that $\phi$ is nonzero everywhere. Furthermore, we also note that the interval
\[
I(\xi):=\{ \theta\in(\arctan^{-1}(\lambda_+), \pi/2)\mid F(\theta;\xi)<0\},
\]
becomes the whole interval $(\arctan^{-1}(\lambda_+), \pi/2)$ as $\xi\to+\infty$, from which we can conclude by a straightforward proof-by-contradiction argument that 
\[
\lim\limits_{\xi\to+\infty}\frac{\phi_\xi}{\phi}=\lambda_+.
\]
If $\mu>r$, then $\lambda_+>0$, which, together with the fact that the interval $(\arctan^{-1}(\lambda_+), \pi/2)$ is forward-invariant for $\theta$, shows that there is no solution to \eqref{e:Lop} in $H^2(\R)$.\\
If $\mu=r(1-b)>r-\frac{c^2}{4d}$, then $\widetilde{\lambda}_+=0>\widetilde{\lambda}_-=-c$. The asymptotic matrix $A_-$ is not hyperbolic but, thanks to the fact that the eigenvalue $\lambda_+=0$ is geometrically simple, there still is an ordinary, but not exponential, dichotomy for the system \eqref{e:LopV} on $\xi\in(-\infty,0]$; see \cite{Cop} for details.
In addition, the fact that $\mu\in(r-\frac{d^2}{4d}, r)$ implies that the asymptotic matrix $A_+$ is hyperbolic with two distinctive negative eigenvalues, yielding an exponential dichotomy with trivial unstable subspace for \eqref{e:LopV} on $\xi\in[0,+\infty)$. Moreover, the analysis based on the polar coordinates still holds. As a result, we conclude that  up to scalar multiplication, there exists a unique solution $\phi$ to \eqref{e:Lop} in $C^2(\R)$ such that 
\[
\lim_{\xi\to+\infty}\frac{\phi_\xi}{\phi}=-\lambda, \qquad \lim_{\xi\to-\infty}\phi=\Upsilon, \quad \textrm{for some }\Upsilon\in\R.
\]
\end{proof}
\end{appendices}

\end{document}